\documentclass[12pt]{article}
\usepackage{amsmath}
\usepackage{amssymb}
\usepackage{harpoon}
\usepackage{latexsym,color}
\usepackage{ntheorem}
\usepackage{mathrsfs}
\usepackage[numbers,sort&compress]{natbib}

  \allowdisplaybreaks[1]
\title{ \Large On the spatially homogeneous
Boltzmann equation for Bose-Einstein particles with balanced potentials}
\author{ Shuzhe Cai\footnote{Department of Mathematical Sciences, Tsinghua
University, Beijing 100084, P.R.China; e-mail address: csz16@mails.tsinghua.edu.cn }
 }

\date{}  

\newcommand{\ld}{\lambda}

\newcommand{\p}{\partial}
\newcommand{\vp}{\varphi}
\newcommand{\vep}{\varepsilon}
\newcommand{\og}{\omega}

\newcommand{\Gm}{\Gamma}
\newcommand{\dt}{\delta}
\newcommand{\Dt}{\Delta}
\newcommand{\fr}{\frac}
\newcommand{\wt}{\widetilde}
\newcommand{\wh}{\widehat}

\newcommand{\bR}{{\mathbb R}^3 }
\newcommand{\bS}{{\mathbb S}^2 }

\newcommand{\bRR}{{\bR}\times{\bR}}

\newcommand{\bRS}{{\bR}\times {\mathbb S}^2 }
\newcommand{\bRRS}{{\bRR}\times{\mathbb S}^2 }

\newcommand{\bRd}{{\mathbb R}^d}
\newcommand{\la}{\langle}
\newcommand{\ra}{\rangle}
\newcommand{\mR}{{\mathbb R}}
\newcommand{\mN}{{\mathbb N}}

\newcommand{\be}{\begin{myequation}}
\newcommand{\ee}{\end{myequation}}
\newcommand{\bes}{\begin{myeqnarray}}
\newcommand{\ees}{\end{myeqnarray}}
\newcommand{\beas}{\begin{eqnarray*}}
\newcommand{\eeas}{\end{eqnarray*}}
\newcommand{\lb}{\label}
\newcounter{thm}
\theoremseparator{.}
\newtheorem{theorem}{Theorem}[section]
\newtheorem{proposition}[theorem]{Proposition}
\newtheorem{definition}[theorem]{Definition}
\newtheorem{lemma}[theorem]{Lemma}
\newtheorem{remark}[theorem]{Remark}

\newtheorem{assumption}[theorem]{Assumption}

\newcounter{myequation}[section]

\newenvironment{proof}{{\bf Proof. }}{$\hfill\Box$}
\newenvironment{myequation}{\stepcounter{myequation}\begin{equation}}{\end{equation}}
\newenvironment{myeqnarray}{\stepcounter{myequation}\begin{eqnarray}}{\end{eqnarray}}
\newcommand{\dnumber}{\stepcounter{myequation}}
\begin{document}

\maketitle
\vskip 0.1in \baselineskip 18.5pt
\begin{abstract}
The paper is concerned with the spatially homogeneous isotropic
Boltzmann equation for Bose-Einstein particles with quantum collision kernel
where the interaction potential $\phi({\bf x})$ can be approximately written
as the delta function plus a certain attractive potential such that the Fourier transform 
$\widehat{\phi}$ of $\phi$
behaves like $0\le \wh{\phi}(\xi)
\le {\rm const.} |\xi|^{\eta}$ for $|\xi|<<1$ for some constant $\eta\ge 1$. We prove that in this case, there is no condensation in finite time for all temperatures and all solutions, and thus it is completely different from the case
$\wh{\phi}(\xi)
\ge {\rm const.}|\xi|^{\eta}$ for $|\xi|<<1$ with $0\le \eta<1/4$ as considered in \cite{Cai-Lu}. For a class of initial data that have some nice integrability near the origin, we also get some regularity, stability and $L^{\infty}$ estimate.
		

\end{abstract}
	\begin{center}\section { Introduction}\end{center}
	We study the spatially homogeneous Boltzmann equation for Bose-Einstein partciles:
\be\fr{\partial}{\partial t}f({\bf v},t)=\int_{{\bRS}}B({\bf
		{\bf v-v}_*},\og)\big(f'f_*'(1+f)(1+f_*)-ff_*(1+f')(1+f_*')\big) {\rm d}\omega{\rm
		d}{\bf v_*} \label{Equation1}\ee
 with $({\bf v}, t)\in{\mathbb R}^3\times(0,\infty)$. This equation (which is now well-known)
 describes time-evolution of a dilute and space homogeneous gas of bosons. Derivations of this equation can be found for instance in \cite{Nordheim},\cite{Uehling and Uhlenbeck},\cite{weak-coupling},\cite{Chapman and Cowling},\cite{ESY}, \cite{LS}.\par
In Equation.(\ref{Equation1}), $f=f({\bf v},t)\ge 0$ is the number density of particles at time $t$ with the velocity ${\bf v}$, and
	$f_*=f({\bf v_*},t), f'=f({\bf v'},t),f_*'=f({\bf v_*'},t)$  where ${\bf v},{\bf v_*}$ and ${\bf v'},{\bf v_*'}$ are velocities of two particles before and after
	their collision and the particle collision is assumed to be elastic:
\be {\bf v}'+{\bf v}_*'={\bf v}+{\bf v}_*,\quad |{\bf v}'|^2+|{\bf v}_*'|^2=|{\bf v}|^2+|{\bf v}_*|^2,\lb{conser}\ee
which can be written as an explicit form:
	\be{\bf v}'={\bf v}- (({\bf v}-{\bf v}_*)\cdot\omega)\omega,\quad {\bf
		v}_*'={\bf v}_*+ (({\bf v}-{\bf v}_*)\cdot\omega)\omega, \qquad \omega\in{\mathbb S}^2
	\lb{colli}\ee.
	
As before we assume that the interation potential $\phi(\cdot)$ of particles is real and is of the central form, i.e.	
$\phi({\bf x})=\phi(|{\bf x}|)$. According to \cite{weak-coupling} and \cite{ESY} in the weak-coupling regime, $B({\bf {\bf v-v}_*},\omega)$ and $\phi$ has the following relation (after normalizing physical parameters)
	\be B({\bf {\bf v-v}_*},\omega)= \fr{1}{(4\pi)^2}|({\bf v-v}_*)\cdot\omega|
	\Phi(|{\bf v}-{\bf v}'|, |{\bf {\bf v}-{\bf v}_*'}|)\label{kernel}\ee where
	\be \Phi(r,\rho)=\big(\widehat{\phi}(r)+\widehat{\phi}(\rho)\big)^2,\quad r,\rho\ge 0\label{kernel2}\ee
	$\widehat{\phi}$ is the Fourier transform of
	$ \phi$:
	$$\widehat{\phi}(r):=\widehat{\phi}(\xi)|_{|\xi|=r}=
	\int_{{\bR}}\phi(|{\bf x}|) e^{-{\rm i}\xi\cdot{\bf x}}{\rm d}{\bf x} \Big|_{|\xi|=r}.$$
	In this paper, the function $r\mapsto \widehat{\phi}(r)$ is often assumed to be continous and bounded on ${\mR}_{\ge 0}$:
\be  \widehat{\phi}\in C_b({\mR}_{\ge 0}). \lb{Phi}\ee
A special case is that $\phi(|{\bf x}|)=\frac{1}{2}\delta({\bf x})$  i.e.
$\widehat{\phi}(r)\equiv \fr{1}{2}$ (hence $\Phi(r)\equiv 1$), where $\delta({\bf x})$ is the three dimensional Dirac delta function concentrating at
	${\bf x}=0$. In this case, (\ref{kernel}) becomes the hard sphere model:
	\be B({\bf
		{\bf v-v}_*},\omega)=\frac{1}{(4\pi)^2}|({\bf v-v}_*)\cdot\omega| \ee
which has been concerned in many papers about  Eq.(\ref{Equation1}).
Since in general $B({\bf {\bf v-v}_*},\omega)$ is a nonnegative Borel function of $|{\bf v-v}_*|$ and $|({\bf v-v}_*)\cdot\omega|$ only, we somtimes also use the notation:
$$ B({\bf {\bf v-v}_*},\omega)\equiv B(|{\bf v-v}_*|,\cos\theta), \quad \quad \theta=\arccos(|({\bf v-v}_*)\cdot\omega|/|{\bf v-v}_*|). $$
Note that by canceling the common terms $f'f_*'ff_*$, Eq.(\ref{Equation1}) becomes
\be\fr{\partial}{\partial t}f({\bf v},t)=\int_{{\bRS}}B({\bf
		{\bf v-v}_*},\og)\big(f'f_*'(1+f+f_*)-ff_*(1+f'+f_*')\big) {\rm d}\omega{\rm
		d}{\bf v_*}\label{Equation-1}\ee
 which still has the cubic nonlinear terms.
Due to the strong nonlinear structure and the effect of condensation, there have been no results on global in time existence of solutions of Eq.(\ref{Equation1}) for the general anisotropic initial data without additional assumptions. See \cite{Briant-Einav} for local in times existence without smallness assumption on the initial data and \cite{LiLu} for  global in times existence with a relative smallness assumption on the initial data .
	For global in time solutions with general initial data, in particular for the case of low temperature, so far one has to consider weak solutions $f$ which are solutions of the
	following equation
	\bes\fr{{\rm d}}{{\rm d}t}\int_{{\bR}}\psi({\bf v})f({\bf v},t){\rm d}{\bf v}&=&
	\fr{1}{2}\int_{{\bRRS}}(\psi+\psi_*-\psi'-\psi'_*)B({\bf {\bf v-v}_*},\omega)f'f_*'{\rm d}{\bf v}{\rm d}{\bf v}_*
	{\rm d}\og  \qquad \nonumber\\
	&
	+&\int_{{\bRRS}}(\psi+\psi_*-\psi'-\psi'_*)B({\bf {\bf v-v}_*},\omega)ff'f_*'{\rm d}{\bf v}{\rm d}{\bf v}_*
	{\rm d}\og  \qquad \label{weak}\ees
	for all test functions $\psi$ and all $t\in [0,\infty)$.
	In general however the cubic integral of \\
$B({\bf v-v}_*,\og)ff'f_*' {\rm d}{\bf v}{\rm d}{\bf v}_*{\rm d}\og$, etc.
are divergent  (see e.g. \cite{Lu2004}).
A subclass of $f$ that has no such divergence is the isotropic (i.e. radially symmetric)
functions: $f({\bf v})=f(|{\bf v}|^2/2)$. By changing variables
$x=|{\bf v}|^2/2, y=|{\bf v}'|^2/2, z=|{\bf v}_*'|^2/2$, one has
$$B({\bf v-v}_*,\og)f(|{\bf v}|^2/2)f(|{\bf v}'|^2/2)
f(|{\bf v}_*'|^2/2) {\rm d}{\bf v}{\rm d}{\bf v}_*{\rm d}\og
=4\pi\sqrt{2}W(x,y,z){\rm d}F(x){\rm d}F(y){\rm d}F(z)$$
where ${\rm d}F(x)=f(x)\sqrt{x}{\rm d}x$, etc., and
	$x,y,z\in {\mR}_{\ge 0}$ in the right side are independent variables. This is the main reason that alomost all results obtained so far are concerned with isotropic
	initial data hence isotropic solutions, see e.g. \cite{Lu2000},\cite{Lu2004},\cite{Lu2005} for the global existence of isotropic solution, moment production and long time weak convergence to equilibrium; \cite{Lu2016},\cite{Lu2018},\cite{Cai-Lu}  for long time strong convergence to the equilibrium; \cite{JPR},\cite{MP2005},\cite{Semikov and Tkachev1},\cite{Semikov and Tkachev2},\cite{SH}  for self-similar structure and deterministic numerical methods; \cite{BV},\cite{EMV2},\cite{EV1},\cite{EV2},\cite{Lu2013} for singular solutions and the formation of blow-up and condensation in finite time; and \cite{A},\cite{AN},\cite{Nouri} for general discussions and basic results for similar models on low temperature evolution of condensation.

Having done researches on the case of hard-sphere like models, the case of other interation models is naturally concerned. Recenetly we found that if the interaction potential $\phi$ is balanced i.e.
$\wh{\phi}(r)=O(r^{\eta})$ for small $r>0$ with $\eta\ge 1$ (see below for details), then there will be
no spontaneous condensation in finite time for all temperatures and all solotions, see Theorem \ref{theorem1.9}.
This is completely different from those of the hard sphere interaction model.

	Before stating the main result of the paper we introduce
	some notations and definitions.
	Let
	$L^1_s({\bR})$ with $s\ge 0$ be the linear space of the weighted Lebesgue integrable functions
	defined by $L^1_0({\bR})=L^1({\bR})$ and
	$$ L^1_s({\bR})=\Big\{f\in L^1({\bR})\,\,\Big|\,\,
	\|f\|_{L^1_s}:=\int_{{\bR}}\la {\bf v}\ra^s|f({\bf v})|{\rm d}{\bf v}<\infty\Big\},\quad \la {\bf v}\ra:=(1+|{\bf v}|^2)^{1/2}.$$
	Let ${\cal B}_k(X)$ ($k\ge 0$) be the linear space
	of signed real Borel measures $F$ on a Borel set $X\subset {\bRd}$ satisfying
	$\int_{X}(1+|x|)^k{\rm d}|F|(x)<\infty$,  where
	$|F|$ is the total variation of $F$.  Let
	$${\cal B}_k^{+}(X)=\{F\in {\cal B}_k(X)\,|\, F\ge 0\}.$$ For the case $k=0$ we also denote
	${\cal B}(X)={\cal B}_{0}(X), {\cal B}^{+}(X)={\cal B}_{0}^{+}(X)$.
	In this paper we only consider two cases $X={\bR}$ and $X={\mR}_{\ge 0}$,
	and in many cases we consider isotropic measures $\bar{F}\in {\cal B}_{2k}({\bR})$, which define and can be defined by
	measures $F\in {\cal B}_{k}({\mR}_{\ge 0})$ in terms of the following relations:
	\be F(A)=\fr{1}{4\pi\sqrt{2}}\int_{{\bR}}{\bf 1}_{
		A}(|{\bf v}|^2/2){\rm d}\bar{F}({\bf v}),\quad A\subset {\mR}_{\ge 0}\lb{m1}\ee
	\be \bar{F}(B)=4\pi\sqrt{2}\int_{{\mR}_{\ge 0}}\Big(\fr{1}{4\pi}\int_{{\bS}}{\bf 1}_{
		B}(\sqrt{2x}\,\og){\rm d}\og\Big){\rm d}F(x),\quad B\subset {\bR} \lb{m2}\ee
	for all Borel measurable sets $A, B$.
	For any $k\ge 0$ let
	\beas\|F\|_k=\int_{{\mathbb R}_{\ge 0}}(1+x)^k{\rm d}|F|(x),\quad F\in {\cal B}_k({\mathbb R}_{\ge 0}).\eeas
		We will also use a semi-norm: $$\|F\|_{1}^{\circ}=\int_{{\mathbb R}_{\ge 0}}x{\rm d}|F|(x).$$ 
Including negative orders, moments for a positive Borel measure $F$ on ${\mathbb R}_{\ge 0}$ are defined by
	\be M_{p}(F) =\int_{{\mathbb R}_{\ge 0}} x^{p}{\rm d}F(x), \qquad p\in(-\infty,\infty).\lb{moment-1}\ee
	Here for the case $p<0$ we adopt the convention $0^{p}=(0+)^p=\infty$, and
	we recall that $\infty\cdot 0=0$. Then it should be noted that
	\be M_{p}(F)<\infty\quad {\rm and}\quad p<0\quad \Longrightarrow\quad F(\{0\})=0.\lb{Mp}\ee
	Moments of orders $0, 1$ correspond to the mass and energy and are particularly denoted as
	\be N(F)=M_0(F)
	, \quad E(F)=M_1(F).\lb{mass-energy}\ee

In order to be able to study long time behavior of solutions of Eq.(\ref{Equation1}) for low temperature,
we first consider weak solutions of the Eq.(\ref{Equation1}) and in fact we could so far only define weak solution for isotropic initial data. A test function space for defining weak solution is chosen
		$$
	C^{1,1}_b({\mathbb R}_{\ge 0})=\Big\{ \varphi\in C^1_b({\mathbb R}_{\ge 0})\, \Big|\,
	\,
	\frac{{\rm d}}{{\rm d}x}\varphi\in {\rm Lip}({\mathbb R}_{\ge 0})\Big\}$$
	For isotropic functions $f=f(|{\bf v}|^2/2)\ge 0$,
	$\vp=\vp(|{\bf v}|^2/2)$ with $f(|\cdot|^2/2)\in L^1_2({\bR}),
	\vp\in  	C^{1,1}_b({\mathbb R}_{\ge 0})$, and for the measure $F$ defined by ${\rm d}F(x)=f(x)\sqrt{x}{\rm d}x$, the collision integrals in (\ref{weak}) can be rewritten

	\beas&&
	\fr{1}{2}\int_{{\bRRS}}(\vp+\vp_*-\vp'-\vp'_*)Bf'f_*'
	{\rm d}{\bf v}{\rm d}{\bf v}_*
	{\rm d}\og=4\pi\sqrt{2}\int_{{\mathbb R}_{\ge 0}^2}{\cal J}[\varphi]{\rm d}^2F,\\ \\
	&&\int_{{\bRRS}}(\vp+\vp_*-\vp'-\vp'_*)B ff'f_*'
	{\rm d}{\bf v}{\rm d}{\bf v}_*
	{\rm d}\og
	=4\pi\sqrt{2}\int_{{\mathbb R}_{\ge 0}^3}{\cal K}[\varphi]{\rm d}^3F
	\eeas
	where $B=B({\bf {\bf v-v}_*},\omega)$ is given by (\ref{kernel}) with
	(\ref{Phi}), ${\rm d}^2F={\rm d}F(y){\rm d}F(z), {\rm d}^3F={\rm d}F(x){\rm d}F(y){\rm d}F(z)$, and
	${\cal J}[\varphi],{\cal K}[\varphi]$ are linear operators
	defined as follows:
	\be{\cal J}[\varphi](y,z)=\frac{1}{2}
	\int_{0}^{y+z}{\cal K}[\varphi](x,y,z)
	\sqrt{x}{\rm d}x,\quad {\cal K}[\varphi](x, y,z)=W(x,y,z)\Delta\varphi(x,y,z)
	,\lb{JKW}\ee
	\be \Delta\varphi(x,y,z)=\varphi(x)+\varphi(x_*)-\varphi(y)-\varphi(z)=
	(x-y)(x-z)
	\int_{0}^1\!\!\!\int_{0}^{1}\vp''(\xi) {\rm d}s {\rm
		d}t
	\lb{diff1}\ee $\xi=y+z-x+t(x-y)+s(x-z)$,
	$x,y,z\ge 0,\, x_*=(y+z-x)_{+}$,
	\be
	W(x,y,z)=\fr{1}{4\pi\sqrt{xyz}}
	\int_{|\sqrt{x}-\sqrt{y}|\vee |\sqrt{x_*}-\sqrt{z}|}
	^{(\sqrt{x}+\sqrt{y})\wedge(\sqrt{x_*}+\sqrt{z})}{\rm d}s
	\int_{0}^{2\pi}\Phi(\sqrt{2}s, \sqrt{2} Y_*){\rm d}\theta
	\qquad {\rm if}\quad x_*xyz>0,\lb{W1}\ee
	\be
	W(x,y,z)=\left\{\begin{array}
		{ll}
		\displaystyle
		\fr{1}{\sqrt{yz}}
		\Phi(\sqrt{2y}, \sqrt{2z}\,)
		\qquad\,\, \qquad\qquad{\rm if}\quad  x=0,\,y>0,\, z>0 \\ \\ \displaystyle
		\fr{1}{\sqrt{xz}}
		\Phi(\sqrt{2x}, \sqrt{2(z-x)}\,)
		\qquad \qquad {\rm if}\quad  y=0,\, z>x>0
		\\ \\ \displaystyle
		\fr{1}{\sqrt{xy}}
		\Phi(\sqrt{2(y-x)}, \sqrt{2x}\,)
		\qquad \quad \quad{\rm if}\quad  z=0,\, y>x>0
		\\ \\ \displaystyle
		0\qquad \quad \qquad {\rm others}\end{array}\right.\lb{W2}\ee
	
	\be Y_*=Y_*(x,y,z,s,\theta)=\left\{\begin{array}
		{ll}\displaystyle
		\bigg|\sqrt{\Big(z-\fr{(x-y+s^2)^2}{4s^2}\Big)_{+}
		}
		+e^{{\rm i}\theta}\sqrt{\Big(x-\fr{(x-y+s^2)^2}{4s^2}
			\Big)_{+}}\,\bigg|\quad {\rm if}\quad s>0\\  \displaystyle
		\\
		0\qquad\qquad  {\rm if}\quad s=0
	\end{array}\right.\lb{Y}\ee
	where $\Phi(r,\rho)$ is given in  (\ref{Phi}), $(u)_{+}=\max\{u, 0\}$,
	$a\vee b =\max\{a,b\},\, a\wedge b =\min\{a,b\},$
	${\rm i}=\sqrt{-1}$.
\vskip2mm
	\begin{remark}\lb{remark0}{\rm It is easily seen that
if $s>0$ and $ |\sqrt{x}-\sqrt{y}|\vee |\sqrt{x_*}-\sqrt{z}| \le s\le (\sqrt{x}+\sqrt{y})\wedge (\sqrt{x_*}+\sqrt{z})$, then
\be x-\fr{(x-y+s^2)^2}{4s^2}\ge0,\quad   z-\fr{(x-y+s^2)^2}{4s^2}\ge 0.\lb{KK0}\ee}
		\end{remark}
\vskip2mm
Based on the existence results (see \cite{Lu2004}),
we introduce directly the concept of measure-valued isotropic solutions of Eq.(\ref{Equation1}) in the weak form:

	\begin{definition}\label{definition1.1} Let $B({\bf {\bf v-v}_*},\omega)$ be given by (\ref{kernel}), (\ref{kernel2}),(\ref{Phi}) and let $F_0\in {\cal B}_{1}^{+}({\mathbb R}_{\ge 0})$. We say that a
		family $\{F_t\}_{t\ge 0}\subset {\cal B}_{1}^{+}({\mathbb R}_{\ge 0})$, or simply $F_t$, is a conservative  measure-valued isotropic solution of Eq.(\ref{Equation1}) on the time-interval $[0, \infty)$ with the initial datum $F_t|_{t=0}=F_0$ if
		
		{\rm(i)} $N(F_t)=N(F_0),\,\, E(F_t)=E(F_0)$ for all $t\in [0, \infty)$,
		
		{\rm(ii)} for every $\varphi\in	C^{1,1}_b({\mathbb R}_{\ge 0})$,  $t\mapsto \int_{{\mathbb R}_{\ge 0}}\varphi(x){\rm d}F_t(x)$ belongs to
		$C^1([0, \infty))$,
		
		{\rm(iii)} for every $\varphi\in 	C^{1,1}_b({\mathbb R}_{\ge 0})$
		\be\frac{{\rm d}}{{\rm d}t}\int_{{\mathbb R}_{\ge 0}}\varphi{\rm
			d}F_t= \int_{{\mathbb R}_{\ge 0}^2}{\cal J}[\varphi]{\rm d}^2F_t+ \int_{{\mathbb R}_{\ge 0}^3}{\cal K}[\varphi]{\rm d}^3F_t\qquad \forall\,t\in[0, \infty). \label{Equation3}\ee
	\end{definition}
	
	\begin{remark}\lb{remark}{\rm  (1) The transition from
			(\ref{W1}) to (\ref{W2}) in defining $W$ is due to
			the identity
			\be (\sqrt{x}+\sqrt{y})\wedge (\sqrt{x_*}+\sqrt{z})-|\sqrt{x}-\sqrt{y}|\vee |\sqrt{x_*}-\sqrt{z}|
			=2\min\{\sqrt{x},\sqrt{x_*},\sqrt{y},\sqrt{z}\}\lb{1.difference}\ee
			from which one sees also that if $\Phi(r,\rho)\equiv 1$,  then $W(x,y,z)$ becomes the function corresponding to the hard sphere model. In the case for hard sphere model, We use notation $W_H$ to replace $W$ and:
		$$W_H(x,y,z)=\fr{1}{\sqrt{xyz}}\min\{\sqrt{x},\sqrt{y},\sqrt{z},\sqrt{x_*}\} $$
			
			(2) By \cite{Lu2004} and Appendix of \cite{Cai-Lu} , we conclude from Theorem 1 (Weak Stability), Theorem 2 (Existence) and Theorem 3 in \cite{Lu2004} that for any $F_0\in {\cal B}_{1}^{+}({\mathbb R}_{\ge 0})$,
			the Eq.(\ref{Equation1}) has always a conservative measure-valued isotropic solution
			$F_t$ on the time-interval $[0, \infty)$ with the initial datum $F_t|_{t=0}=F_0$.}
	\end{remark}
	In order to get regularity results and $L^{\infty}$ estimates, we also need the definition of mild solutions as follows.
	\begin{definition}\label{definition 1-1}
		Let  $B({\bf {\bf v-v}_*},\omega)$ be given by (\ref{kernel}), (\ref{kernel2}),(\ref{Phi}). Let $f(x,t)$ be a
		nonnegative measurable function on ${\mR}_{\ge 0}\times [0, T_{\infty})\,(0<T_{\infty}\le \infty)$.
		We say that $f(\cdot,t)$
		is a mild solution of Eq.(\ref{Equation1}) on the time-interval $[0, T_{\infty})$ if
		$f$ satisfies
		
		{\rm (i)} $\sup\limits_{t\in [0,T]}\int_{{\mathbb R}_{+}}(1+x)f(x,t)\sqrt{x}{\rm d}x<\infty\quad \forall\, 0<T<T_{\infty},$
		
		{\rm (ii)} there is a null set $Z\subset {\mathbb R}_{\ge 0}$ which is independent of $t$ such that for all  $x\in {\mathbb R}_{\ge 0}\setminus Z$ and all $t\in[0, T_{\infty})$,
		$\int_{0}^{t}{\rm d}\tau\int_{{\mathbb R}_{\ge 0}^2}W(x,y,z)[ f'f_*'(1+f+f_*))+
		ff_*(1+f'+f_*')]\sqrt{y}\sqrt{z}{\rm d}y{\rm
			d}z<\infty$ and
		$$
		f(x,t)=f_0(x) +\int_{0}^{t}Q(f)(x,\tau){\rm d}\tau.$$
		Here $f_0=f(\cdot,0)$ denotes the initial datum of $f(\cdot,t)$ and
		$$Q(f)(x)=\int_{{\mR}_{\ge 0}^2}W(x,y,z)[ f'f_*'(1+f+f_*))-
		ff_*(1+f'+f_*')]\sqrt{y}\sqrt{z}{\rm d}y{\rm d}z.$$
		
	\end{definition}
	It is obvious that if $f(\cdot,t)$ is a mild solution of Eq.(\ref{Equation1}), then the measure $F_t$, defined by ${\rm d}F_t(x)=f(x,t)\sqrt{x}{\rm d}x$, is a distributional solution of Eq.(\ref{Equation1}).
\vskip2mm
	
{\bf Kinetic Temperature.}   Let $F\in {\mathcal B}_{1}^{+}({\mathbb R}_{\ge 0})$, $N=N(F), E=E(F)$ and suppose $N>0$. If $m$ is the mass of one particle,
	then  $m4\pi\sqrt{2}N$, $m 4\pi\sqrt{2}E$ are total
	mass and kinetic energy of the particle system per unite
	space volume.
	The kinetic temperature $\overline{T}$ and the kinetic
	critical temperature $\overline{T}_c$ are defined by (see e.g.\cite{Lu2004} and references therein)
	$\overline{T}=\frac{2m}{3k_{\rm B}}
	\frac{E}{N}
	,\,\overline {T}_c=\frac{\zeta(5/2)}{(2\pi)^{1/3}[\zeta(3/2)]^{5/3}}
	\frac{2m}{k_{\rm B}}N^{2/3}$
	where $k_{\rm B}$ is the Boltzmann constant, $\zeta(s)=\sum_{n=1}^{\infty} n^{-s}, s>1$.
	Keeping in mind the constant $m4\pi\sqrt{2}$,
	there will be no confusion if we also call $N$ and $E$ the
	mass and energy of a particle system.
	\vskip2mm
	
	{\bf Regular-Singular Decomposition.}  According to measure theory (see e.g.\cite{Rudin}), every
	finite positive Borel measure can be uniquely decomposed into
	regular part and singular part with respect to the Lebesgue measure. For instance
	if $F\in {\cal B}_1^{+}({\mathbb R}_{\ge 0})$, then
	there exists unique $0\le f\in L^1({\mathbb R}_{\ge 0},(1+x)\sqrt{x}{\rm d}x)$, $\nu\in{\cal B}_1^{+}({\mathbb R}_{\ge 0})$ and a Borel set $Z\subset {\mathbb R}_{\ge 0}$
	such that
	$${\rm d}F(x)=f(x)\sqrt{x}{\rm d}x+{\rm d}\nu(x),\quad mes(Z)=0,\quad \nu({\mR}_{\ge 0}\setminus Z)=0.$$
	We call $f$ and $\nu$ the regular part and the singular part of $F$
	respectively\footnote{Strictly speaking the product $f(x)\sqrt{x}$
		is the regular part of $F$. The reason that we only mention $f$ is because $f(x)\sqrt{x}$ comes from the 3D-isotropic function $f=f(|{\bf v}|^2/2)$. }.
	\vskip2mm
	
	{\bf Bose-Einstein Distribution.}  According to Theorem 5 of \cite{Lu2004} and
	its equivalent version proved in the Appendix of \cite{Lu2013} we know that
	for any  $N>0$, $E>0$ the Bose-Einstein distribution $F_{{\rm be}}\in {\mathcal B}_1^{+}({\mathbb R}_{\ge 0})$, which is the unique equilibrium solution of Eq.(\ref{Equation3})
	satisfying $N(F_{\rm be})=N, E(F_{\rm be})=E$, is given by
	\be {\rm d}F_{\rm be}(x)=
	\left\{\begin{array}{ll}
		\displaystyle \frac{1}{Ae^{x/\kappa}-1}\sqrt{x}{\rm d}x,\quad A>1,\,\,\,\,\, \qquad
\qquad \qquad \qquad {\rm if}
		\quad \overline{T}/\overline{T}_c>1,\\ \\
		\displaystyle
		\frac{1}{e^{x/\kappa}-1}\sqrt{x}{\rm d}x+\big(1-(\overline{T}/\overline{T}_c)^{3/5}\big)
	N\dt(x){\rm d}x, \qquad  {\rm if}\quad
		\overline{T}/ \overline{T}_c\le 1
	\end{array}\right.\label{1-E}\ee
where $\dt(x)$ is the Dirac
	delta function concentrated at $x=0$, and
functional relations of the coefficients $A=A(N,E), \kappa=\kappa(N,E)$
can be found in for instance Proposition 1 in \cite{Lu2005}. The positive number $(1-(\overline{T}/\overline{T}_c)^{3/5} )N$
	is called the Bose-Einstein condensation (BEC) of the equilibrium state of Bose-Einstein particles at low temperature $\overline{T}<\overline{T}_c$.
	\vskip2mm
	{\bf Entropy.}  The entropy functional for Eq.(\ref{Equation1}) is
	\be S(f)=\int_{{\bR}}\big((1+f({\bf v}))\log(1+f({\bf v}))
	-f({\bf v})\log f({\bf v})\big){\rm d}{\bf v},\quad 0\le f\in L^1_2({\bR}).\lb{ent}\ee
As in \cite{Cai-Lu}, we define the entropy $S(F)$ of a measure
	$F\in {\cal B}_2^{+}({\bR})$ by
	\be S(F):=\sup_{\{f_n\}_{n=1}^{\infty}}\limsup_{n\to\infty}S(f_n)\lb{ent1}\ee
	where $\{f_n\}_{n=1}^{\infty}$ under the sup is taken all sequences in $L^1_2({\bR})$ satisfying
	\bes&& f_n\ge 0,\quad \sup_{n\ge 1}\|f_n\|_{L^1_2}<\infty;  \lb{1.20}\\
	&&\lim_{n\to\infty}\int_{{\bR}}\psi({\bf v})f_n({\bf v}){\rm d}{\bf v}
	=\int_{{\bR}}\psi({\bf v}){\rm d}F({\bf v})\qquad \forall\,\psi\in C_b({\bR}).\dnumber\lb{1.21}\ees
	Let $0\le f\in L^1_2({\bR})$ be the regular part of $F$, i.e.
	${\rm d}F({\bf v})=f({\bf v}){\rm d}{\bf v}+{\rm d}\nu({\bf v})$ with $\nu\ge 0$
	the singular part of $F$. By Lemma 3.2 of \cite{Cai-Lu} we have
	\be S(F)=S(f)\lb{ent2}\ee
	which shows that the singular part of
	$F$ has no contribution
	to the entropy $S(F)$ and that $F$ is non-singular if and only if $S(F) > 0.$
	For any
	$0\le f\in L^1({\mathbb R}_{\ge 0},(1+x)\sqrt{x}\,{\rm d}x)$, the entropy
	$S(f)$ is defined by $S(f)=S(\bar{f})$
	with $\bar{f}({\bf v}):=f(|{\bf v}|^2/2)$, so that  (using (\ref{ent}) and change of variable)
	\be S(f)=S(\bar{f})=4\pi\sqrt{2}\int_{{\mathbb R}_{\ge 0}}\big((1+f(x))\log(1+f(x))-f(x)\log f(x)\big)\sqrt{x}\,{\rm d}x.\lb{ent3}\ee
	In general, the entropy $S(F)$  for a measure
	$F\in {\cal B}_1^{+}({\mathbb R}_{\ge 0})$ is defined by
	$S(F)=S(\bar{F})$ where
	$\bar{F}\in {\cal B}_2^{+}({\bR})$ is defined by $F$ through (\ref{m2}) and
	$S(\bar{F})$ is defined by (\ref{ent1}) or (\ref{ent2}) .

In a recent work \cite{Cai-Lu}, condensation in finite time and strong convergence to equilibrium of $F_t$ have been proven for the collision kernel $B$ that is similar to hard sphere model. We summarize them as the following theorem.
	
\begin{theorem}[\cite{Cai-Lu}]\lb{theorem-EV}
Suppose $B({\bf {\bf v-v}_*},\omega)$  is given by
(\ref{kernel}),(\ref{kernel2}) where the Fourier transform
$r\mapsto \widehat{\phi}(r)$ (of a radially symmetric interaction potential
${\bf x}\mapsto \phi(|{\bf x}|)$) is continuous and non-decreasing on ${\mR}_{\ge 0}$,
and there are constants $0<b_0\le 1/2, 0\le \eta<\fr{1}{4}$ such that
\be b_0\fr{r^{\eta}}{1+r^{\eta}}\le \widehat{\phi}(r)\le \fr{1}{2}\quad\forall\, r\ge 0.\lb{1.5}\ee
Let $F_0\in {\mathcal B}_1^{+}({\mathbb R}_{\ge 0})$ satisfy
$N(F_0)>0, E(F_0)>0$, let $F_{\rm be}$ be
the unique Bose-Einstein distribution with the same mass $N=N(F_0)$ and energy $E=E(F_0)$, and let
$\fr{1}{20}<\ld<\fr{1}{19}$.  Then
there exists a conservative  measure-valued isotropic solution $F_t$ of
Eq.(\ref{Equation1}) on $[0,\infty)$ with the initial datum $F_0$ such that $S(F_{\rm be})\ge S(F_t)\ge S(F_0), S(F_t)>0 $ for all $t>0$ and
\beas S(F_{\rm be})-S(F_t)\le C(1+t)^{-\ld},\quad\|F_t-F_{{\rm be}}\|_1\le C(1+t)^{-\fr{(1-\eta)\ld}{2(4-\eta)}}\qquad \forall\, t\ge 0.\eeas
In particular if\, $\overline{T}/\overline{T}_c<1$  then
\be\big|F_t(\{0\})-(1-(\overline{T}/\overline{T}_c)^{3/5})N\big|\le  C(1+t)^{-\fr{(1-\eta)\ld}{2(4-\eta)}}\qquad \forall\, t\ge 0.\label{rate2}\ee
Here the constant $C>0$ depends only on $N,E, b_0, \eta$ and $\ld.$

\end{theorem}

 Theorem \ref{theorem-EV} tells us that if the Fourier transform of the interaction potential  $ \widehat{\phi}(r)$ satisfies (\ref{1.5}) (which may be viewed as a small perturbation of the hard sphere model $\widehat{\phi}(r)\equiv 1/2$), finite time condensation and strong convergence to equilibrium hold true. It is natural to ask whether or not they are still true for an opposite case
 $0\le \widehat{\phi}(r)\le b_0\fr{r^{\eta}}{1+r^{\eta}}$ ?
 In this paper we show that this is false if $\eta\ge 1$. More precisely, we introduce the following Assumption.
\vskip2mm
\begin{assumption}\lb{assp}   The collision kernel $B({\bf {\bf v-v}_*},\omega)$ is given by
	(\ref{kernel}),(\ref{kernel2}), where the Fourier transform
	$r\mapsto \widehat{\phi}(r)$ (of a radially symmetric interaction potential) is continuous on $[0,\infty)$,
	and there are constants $b_0>0$, $\eta\ge1$ such that
	\be 0\le \widehat{\phi}(r)\le b_0\fr{r^{\eta}}{1+r^{\eta}}\quad\forall\, r\ge 0\lb{1.6}\ee
	and there is  a function $k\in C^1([1,\sqrt{2}])$ with $k(1)=1$, such that
$$\widehat{\phi}(ar)\le k(a) \widehat{\phi}(r)\quad \forall\, r>0,\,\,\forall\, 1<a\le \sqrt{2}. $$

\end{assumption}
In this paper we always denote $q_1:=\max\limits_{x\in [1,\sqrt{2}]}\max \{2k(x)k'(x), 0\}.$ \par
If $\wh{\phi}$ satisfies Assumption \ref{assp}, then we say that $\phi$ is a balanced potential since Assumption \ref{assp} implies that
$$\int_{\mR^3}\phi(|{\bf x}|) {\rm d}{\bf x}=\widehat{\phi}(0)=0. $$
Generally if $\eta>n\in {\mN}$ and $(1+|{\bf x}|^{n})\phi \in L^1({\mR}^3)$, then
    $$\int_{\mR^3}x_1^{\alpha_1}x_2^{\alpha_2}x_3^{\alpha_3}\phi(|{\bf x}|) {\rm d}{\bf x}={\rm i}^{|\alpha|}D^{\alpha}\wh{\phi}(0)=0$$
    for all indices $\alpha$ with $|\alpha|=\alpha_1+\alpha_2+\alpha_3\le n.$
    Roughly speaking, the higher $\eta$ is, the more balnaced $\phi$ becomes.

  There are many examples of balanced potentials. For instance $\phi(|{\bf x}|)=\fr{1}{2}(\delta({\bf x})-U(|{\bf x}|))$ where $U(|{\bf x}|)\ge 0$ is 3D Yukawa potential $U(|{\bf x}|)=\fr{1}{4\pi |{\bf x}|}e^{-|{\bf x}|}$, ${\bf x}\in {\bR}$, then  $\widehat{\phi}(r)=\fr{r^2}{1+r^2}$ satsifies Assumption \ref{assp}. More generally, given any $\eta>\fr{3}{2}$, $g(r)=\fr{1}{1+r^\eta}$,.$\eta>\fr{3}{2}$ implies  $g\in L^2(\mR^3)$, then one can use basic knowledge of Fourier tranform to get a function $U_{\eta}\in L^2(\mR^3)$ such that $\widehat U_{\eta}(r)=\fr{1}{1+r^\eta}$, so that $\widehat{\phi}(r)=\fr{r^\eta}{1+r^\eta}$ satisfies Assumption \ref{assp}.
\vskip2mm

    {\bf Main Results.} The main results of the paper is as follows:

		\begin{theorem}\lb{theorem1.9}
		Suppose $B({\bf {\bf v-v}_*},\omega)$ satisfy  Assumption \ref{assp}.
		Let $F_0\in {\mathcal B}_1^{+}({\mathbb R}_{\ge 0})$  with mass $N=N(F_0)>0$ and energy $E=E(F_0)>0$ and let
		$F_t$ be a conservative  measure-valued isotropic solution $F_t$ of
		Eq.(\ref{Equation1}) on $[0,\infty)$ with the initial datum $F_0$ (the existence of $F_t$
		has been insured by Remark \ref{remark}). Then we have
		$$F_t(\{0\}) \le e^{ct} F_0(\{0\})\qquad \forall\, t\ge 0$$
		where $c=8^{2+\eta}b_0^2(1+q_1)N^2$.
		In particular if $F_0(\{0\})=0$, then $F_t(\{0\})=0$ for all $t\ge 0$.
	\end{theorem}
	
\begin{theorem} \lb{theorem1.10}
	Suppose $B({\bf {\bf v-v}_*},\omega)$ satisfy  Assumption \ref{assp} with $\eta\ge \fr{3}{2}$. Let  $F_t$ be a conservative  measure-valued isotropic solution solution of Eq.(\ref{Equation1}) on $[0, \infty)$ whose
	initial datum $F_0$ is regular and satisfies $M_{-1/2}(F_0)<\infty$,
	Then
	$F_t$ is regular for all $t\in [0, \infty)$ and its density
	$f(\cdot,t)$  is a mild solution of Eq.(\ref{Equation1}) on
	$[0, \infty)$ satisfying
	$f\in C([0, \infty); L^1({\mathbb R}_{+}))$ and $f(\cdot,0)=f_0$,
	where $f_0$ is the density of $F_0$. In particular if $F_t$ is conservative, so is $f(\cdot,t)$ on
	$[0, \infty)$.
\end{theorem}
For any given $F_0\in  {\cal B}_{1}^{+}({\mathbb R}_{\ge 0})$ we define a function
$\Psi_{F_0}(\vep)$ on $\vep\in[0,\infty)$ by
\be\Psi_{F_0}(\vep)=\vep+\sqrt{\vep}+\int_{\fr{1}{\sqrt{\vep}}}^{\infty}x {\rm d}F_0(x),\quad \vep>0;\quad
\Psi_{F_0}(0)=0.\lb{Psi}\ee
Here $\int_{\fr{1}{\sqrt{\vep}}}^{\infty}$ can be understood as either
$\int_{(\fr{1}{\sqrt{\vep}}, \infty)}$ or $\int_{[\fr{1}{\sqrt{\vep}}, \infty)}$. Now we can introduce stability theorem.
\begin{theorem}\lb{theorem1.11}
	Suppose $B({\bf {\bf v-v}_*},\omega)$ satisfy  Assumption \ref{assp} with $\eta\ge \fr{3}{2}$, moreover assume $\widehat{\phi}$ satisfy  $\widehat{\phi}(r)\ge a_0r^{-\beta}>0$ for all $r\ge R$ with $R\ge 0$ and $0\le \beta<\fr{1}{2}$. Let $F_0, G_0\in {\cal B}_1^{+}({\mathbb R}_{\ge 0})$ satisfy $M_{-1/2}(F_0)<\infty, M_{-1/2}(G_0)<\infty$, and
	$F_t, G_t$ be conservative  measure-valued isotropic solution to Eq.(\ref{Equation1}) on $[0, \infty)$  with their
	initial data $F_0, G_0$ respectively. Then
	\be \|F_t-G_t\|_1\le C\Psi_{F_0}(\|F_0-G_0\|_1) e^{e^{ct}}\qquad
	\forall\, t\in [0, \infty)\lb{H0}\ee
	where $\Psi_{F_0}(\cdot)$ is defined in (\ref{Psi}) and
	$C,c$ are finite positive constants depending only on $N(F_0), E(F_0),N(G_0),E(G_0)$, $a_0$,$b_0$,$\beta$,$\eta$,$q_1$,$R$,$M_{-1/2}(F_0)$, $M_{-1/2}(G_0)$.
	
	In particular if $F_0=G_0$, then
	$F_t=G_t$ for all $t\in [0,\infty).$
\end{theorem}	
\begin{theorem}\lb{theorem1.12}
	Let the collision kernel $B({\bf {\bf v-v}_*},\omega)$ be given by
	(\ref{kernel}),(\ref{kernel2}), (\ref{Phi}) and suppose  $\widehat{\phi}(r)$ satisfy
	\be a(r)\le \widehat{\phi}(r) \quad\forall\, r\ge  0.\lb{laa}\ee
	where $a(\cdot)$ is a non-decreasing continuous function on $[0, \infty)$  satisfying $a(r)>0 $ for $r>0$ and $b_1 $ is a constant.	Given any  $F_0\in {\mathcal B}_1^{+}({\mathbb R}_{\ge 0})$  with $N=N(F_0)>0$ and $E=E(F_0)>0$ and let $F_t$ be a conservative  measure-valued isotropic solution $F_t$ of
	Eq.(\ref{Equation1}) on $[0,\infty)$ with the initial datum $F_0$ (the existence of $F_t$
	has been insured by Remark \ref{remark}). Let $F_{\rm be}$ be
	the unique Bose-Einstein distribution with the same mass and energy as $F_0$. Then
	$$\lim_{t\to\infty}S(F_t)=S(F_{\rm be}),\quad \lim_{t\to\infty}\|F_t-F_{\rm be}\|_{1}^{\circ}=0.$$
	Conserquently it holds the weak convergence:
	$$\lim_{t\to\infty}\int_{{\mR}_{\ge 0}}\vp(x){\rm d}F_t(x)=\int_{{\mR}_{\ge 0}}\vp(x){\rm d}F_{\rm be}(x)
	\quad \forall\, \vp\in C_b({\mR}_{\ge 0}).$$
\end{theorem}	

\quad The rest of the paper is organized as follows. In Section 2 we prove  Theorem \ref{theorem1.9} and Theorem \ref{theorem1.3}: non-condensation in finite time and propagation of $M_{-p}(F_t)$ for $0<p\le \fr{1}{2}$. In Section 3, we prove moment production, positive lower bound of entropy and weak convergence. In Section 4, we use propagation of $M_{-1/2}(F_t)$ to get regularity,stability(uniqueness) of $F_t$ if $M_{-1/2}(F_0)<\infty$. We also prove the global existence of   mild solution and strong solution of Eq.(\ref{Equation1}) if $M_{-1/2}(F_0)<\infty$ and get $L^{\infty}$ estimate about the mild solutions.

\begin{center}\section { Non-condensation in finite time and propation of negative order of moment }\end{center}
																 \quad In this section we prove non-condensation in finite time and propagation of  $M_{-p}(F_t)<\infty$ for $0<p\le \fr{1}{2}$. To prove them, we  need the following lemma about $W(x,y,z)$.
\begin{lemma}\lb{lemma0} Suppose $B({\bf {\bf v-v}_*},\omega)$ satisfy  Assumption \ref{assp} , then the following estimates about $W(x,y,z)$ hold
	\bes &&W(x,y,z)\le 4b_0^2\fr{\min\{1,\max\{8x,8y,8z\}^\eta\}}{\sqrt{x}\sqrt{y}\sqrt{z}}\min\{\sqrt{x},\sqrt{y},\sqrt{z},\sqrt{x_*}\}, \quad \forall x,y,z>0  \lb{W01} \\
&&W(0,y,z)\le 4b_0^2 \fr{\min\{1,\max\{8y,8z\}^\eta\}}{\sqrt{yz}}, \quad y,z>0 \dnumber\lb{W02}\\
&&W(x,0,z)\le \fr{4b_0^2}{\sqrt{xz}}\min\{1,(8z)^\eta\}, \quad z>x>0 \dnumber\lb{W03}\\
&&W(x,y,0)\le\fr{4b_0^2}{\sqrt{xy}}\min\{1,(8y)^\eta\}, \quad y>x>0\dnumber\lb{W04}\\
&&W(x,y,z)\le (1+q_1\fr{y}{z})W(y,x,z), \quad \forall 0\le x\le y\le \fr{z}{2}\dnumber\lb{W05} \ees
where $b_0$,$\eta$ and $q_1$ are defined in Assumption \ref{assp}.
\end{lemma}

\par
\noindent\begin{proof}First we need to estimate  $\Phi(\sqrt{2}\,s, \sqrt{2}\,Y_*)$.
By  (\ref{Y}) and (\ref{KK1}), for the case of $ |\sqrt{x}-\sqrt{y}|\vee |\sqrt{x_*}-\sqrt{z}| \le s\le (\sqrt{x}+\sqrt{y})\wedge (\sqrt{x_*}+\sqrt{z}),s>0$,we have
 $$s\le 2\max\{\sqrt{x},\sqrt{y},\sqrt{z}\}, $$
$$Y_*\le \sqrt{z-\fr{(x-y+s^2)^2}{4s^2}
}
+\sqrt{x-\fr{(x-y+s^2)^2}{4s^2}
}\
\le 2\max\{\sqrt{x},\sqrt{y},\sqrt{z}\}.
$$

\noindent So we obtain
\beas &&\Phi(\sqrt{2}\,s, \sqrt{2}\,Y_*)
=\Big(\widehat{\phi}(\sqrt{2}s+\widehat{\phi}(\sqrt{2}Y_*)
\Big)^2 \le  b_0^2\Big(\fr{(\sqrt{2}s)^{\eta}}{1+(\sqrt{2}s)^{\eta}}+\fr{(\sqrt{2}Y_*)^{\eta}}{1+(\sqrt{2}Y_*)^{\eta}})
\Big)^2\\
&&\le  b_0^2\Big(\fr{(2\sqrt{2}\max\{\sqrt{x},\sqrt{y},\sqrt{z}\})^{\eta}}{1+(2\sqrt{2}\max\{\sqrt{x},\sqrt{y},\sqrt{z}\})^{\eta}}+\fr{(2\sqrt{2}\max\{\sqrt{x},\sqrt{y},\sqrt{z}\})^{\eta}}{1+(2\sqrt{2}\max\{\sqrt{x},\sqrt{y},\sqrt{z}\})^{\eta}})
\Big)^2\\ &&\le 4b_0^2 \min\{1,\max\{8x,8y,8z\}^\eta\}.  \eeas
Together with (\ref{W2}), (\ref{1.difference}), (\ref{1.6}), this yields
\beas W(x,y,z)
&=&\fr{1}{4\pi\sqrt{xyz}}
\int_{|\sqrt{x}-\sqrt{y}|\vee |\sqrt{x_*}-\sqrt{z}|}
^{(\sqrt{x}+\sqrt{y})\wedge(\sqrt{x_*}+\sqrt{z})}{\rm d}s
\int_{0}^{2\pi}\Phi(\sqrt{2}s, \sqrt{2} Y_*){\rm d}\theta
\\
&\le& \fr{b_0^2}{\pi \sqrt{xyz}}\int_{|\sqrt{x}-\sqrt{y}|\vee |\sqrt{x_*}-\sqrt{z}|}
^{(\sqrt{x}+\sqrt{y})\wedge(\sqrt{x_*}+\sqrt{z})}{\rm d}s
\int_{0}^{2\pi}\min\{1,\max\{8x,8y,8z\}^\eta\}{\rm d}\theta\\
&=&4b_0^2\fr{\min\{1,\max\{8x,8y,8z\}^\eta\}}{\sqrt{xyz}}\min\{\sqrt{x},\sqrt{y},\sqrt{z},\sqrt{x_*}\} .\qquad \forall x,y,z>0 \eeas

  \noindent Thus we complete the proof of (\ref{W01}). The proofs of (\ref{W02}),(\ref{W03}),(\ref{W04}) are analogous.
\par
  In order to  prove (\ref{W05}),first we prove an useful inequality:
  \bes 1\le \fr{Y_*}{Y_*^{\sharp}}\le \fr{\sqrt{z}}{\sqrt{z-y}}\quad \forall 0\le x\le y< z   \lb{Y*}\ees
  where $Y_*^{\sharp}=Y_*(y,x,z,s,\theta)$. To prove this inequality, recalling that
  \be Y_*=Y_*(x,y,z,s,\theta)=\left\{\begin{array}
{ll}\displaystyle
\bigg|\sqrt{\Big(z-\fr{(x-y+s^2)^2}{4s^2}\Big)_{+}
}
+e^{{\rm i}\theta}\sqrt{\Big(x-\fr{(x-y+s^2)^2}{4s^2}
\Big)_{+}}\,\bigg|\quad {\rm if}\quad s>0\\  \displaystyle
\\
0\qquad\qquad  {\rm if}\quad s=0,

\nonumber \end{array}\right.\ee
we have $\fr{Y_*}{Y_*^{\sharp}}= \fr{\sqrt{z}}{\sqrt{z-y}}$ for  $0=x\le y<z$. For the case $x\neq0$ ,denote  $u=\fr{(x-y+s^2)^2}{4s^2}$,$U=z-\fr{(x-y+s^2)^2}{4s^2}$,$V=z-\fr{(y-x+s^2)^2}{4s^2}=U-y+x $ ,                     $O=x-\fr{(x-y+s^2)^2}{4s^2}=y-\fr{(y-x+s^2)^2}{4s^2}$.
By (\ref{KK0}) we know $u\le x$ if $s\in[\sqrt{y}-\sqrt{x}, \sqrt{x}+\sqrt{y}]$,$0< x \le y \le z$, thus we obtain
$$\fr{Y_*}{Y_*^{\sharp}}=\fr{|\sqrt{U}+e^{i\theta}\sqrt{O}|}{|\sqrt{V}+e^{i\theta}\sqrt{O}|}=\fr{\sqrt{U+O+\sqrt{UO}(e^{i\theta}+e^{-i\theta})}}{\sqrt{V+O+\sqrt{VO}(e^{i\theta}+e^{-i\theta})} }\le \sqrt{\fr{U}{V}} =  \fr{\sqrt{z-u}}{\sqrt{z+(x-y)-u}} \le \fr{\sqrt{z-x}}{\sqrt{z-y}} \le \fr{\sqrt{z}}{\sqrt{z-y}}$$
and
$$\fr{Y_*}{Y_*^{\sharp}}=\fr{\sqrt{U+O+\sqrt{UO}(e^{i\theta}+e^{-i\theta})}}{\sqrt{V+O+\sqrt{VO}(e^{i\theta}+e^{-i\theta})} }\ge 1 $$
Now we are ready to prove inequality  (\ref{W05});using (\ref{Y*}), (\ref{kernel2}) , (\ref{W1}), (\ref{W2}), it suffices to prove
  $$\widehat{\phi}(\sqrt{2}\,Y_*)\le \sqrt{1+q_1\fr{y}{z}}\widehat{\phi}(\sqrt{2}\,Y_*^{\sharp}) \quad  \forall 0\le x\le y\le \fr{z}{2} $$
 If $\widehat{\phi}(\sqrt{2}Y_*^{\sharp})=0$, by  Assumption \ref{assp}, this inequality is obvious. If $\widehat{\phi}(\sqrt{2}Y_*^{\sharp})\neq 0$
   \beas&&(\fr{\widehat{\phi}(\sqrt{2}Y_*)}{\widehat{\phi}(\sqrt{2}Y_*^{\sharp})})^2\le (k(\fr{Y_*}{Y_*^{\sharp}}))^{2}\le 1+(\fr{Y_*}{Y_*^{\sharp}}-1)\max_{x\in [1,\fr{Y_*}{Y_*^{\sharp}}]}2k(x)k'(x) =1+(\fr{Y_*}{Y_*^{\sharp}}-1)\max_{x\in [1,\fr{Y_*}{Y_*^{\sharp}}]}2k(x)k'(x) \\
   &&\le 1+(\sqrt{\fr{z}{z-y}}-1)\max_{x\in [1,\sqrt{\fr{z}{z-y}}]}2k(x)k'(x)\le 1+(\sqrt{\fr{z}{z-y}}-1)q_1\le 1+q_1 \fr{y}{2(z-y)}\le 1+q_1\fr{y}{z}  \eeas
   for all $0\le x\le y< \fr{z}{2}$, so we complete the proof.
\end{proof}
\vskip2mm
	\begin{remark}\lb{remark1}
	{\rm	Combining with \cite{Cai-Lu}, we know for $\widehat{\phi}(r)=b_0\fr{r^{\eta}}{1+r^{\eta}}$ the following esimates hold $$\fr{b_0^2}{8}\fr{z^\eta}{\sqrt{y}\sqrt{z}}\le W(x,y,z)\le 4b_0^2\fr{(8z)^\eta}{\sqrt{y}\sqrt{z}} \quad \forall 0< x\le y\le z\le 1. $$
		In this case, $W$ is still unbounded.}
		\end{remark}
{\bf Proof of Theorem \ref{theorem1.9}.} Denote $\varphi_{\vep}(x)=(1-\fr{x}{\vep})_+^2$. By the definition of weak solution,
			\bes \int_{{\mathbb R}_{\ge 0}}\varphi_{\vep}(x){\rm d}F_t(x)&=&\int_{{\mathbb R}_{\ge 0}}\varphi_{\vep}(x){\rm d}F_0(x)\nonumber \\&+&\int_0^t {\rm d}\tau\int_{{\mathbb R}_{\ge 0}^2}{\cal J}[\vp_{\vep}](y,z){\rm d}F_{\tau}(y){\rm d}F_{\tau}(z)\nonumber \\&+&\int_0^t d\tau \int_{{\mathbb R}_{\ge 0}^3}{\cal K}[\varphi_{\vep}]{\rm d}^3F_{\tau}.		\lb{I1}\ees
By the fact that  $W(x,y,z)\le 4b_0^2W_H(x,y,z)$, we have
		 \beas&&{\cal J}[\vp_{\vep}](y,z)\le \frac{1}{2}
		\int_{0}^{y+z}W(x,y,z)(\vp_{\vep}(x)+\vp_{\vep}(y+z-x))
		\sqrt{x}{\rm d}x\\
	&&	\le 2b_0^2\int_{0}^{y+z}W_H(x,y,z)(\vp_{\vep}(x)+\vp_{\vep}(y+z-x))
		\sqrt{x}{\rm d}x=4b_0^2\int_{0}^{y+z}W_H(x,y,z)\vp_{\vep}(x)
		\sqrt{x}{\rm d}x.\eeas
		Combining with the fact that $W_H(x,y,z)\sqrt{x}\le \sqrt{\fr{2}{y+z}}$ for all $0<x<y+z$ and $\sup_{r>0}\sqrt{\fr{1}{r}}\int_0^{r} \vp_1(x){\rm d}x\le 1$, this leads to
		\bes\int_{R^2 \ge 0}{\cal J}[\vp_{\vep}](y,z){\rm d}^2F_t&\le& 	4\sqrt{2}b_0^2 \sqrt{\vep} \int_{y,z\ge 0,y+z>0}\sqrt{\fr{\vep}{y+z}}{\rm d}F_t(y){\rm d}F_t(z)\int_0^{\fr{y+z}{\vep}} \vp_1(x){\rm d}x \nonumber \\
		&\le& 4\sqrt{2}b_0^2N^2\sqrt{\vep} \lb{J^+}.\ees
		The term $ \int_{{\mathbb R}_{\ge 0}^3}{\cal K}[\vp_{\vep}]{\rm d}^3F_{\tau}$ can be decomposed into the following parts (see \cite{Cai-Lu}):
		\bes \int_{{\mathbb R}_{\ge 0}^3}{\cal K}[\vp_{\vep}]{\rm d}^3F_{\tau}\nonumber
		&=&\left(2\int_{0\le x<y<z}+2\int_{0\le y<x<z}+
		\int_{0\le x<y=z}+\int_{0\le y, z<x}
		\right)W(x,y,z)\Dt\vp_{\vep}(x,y,z){\rm d}^3F_{\tau}\nonumber
		\\
		&=&\int_{0\le x<y\le z}\chi_{y,z}W(x,y,z)\Dt_{\rm sym}\vp_{\vep}(x,y,z){\rm d}^3F_{\tau}\nonumber\\
&+&
		2\int_{0\le x<y<z}\big(W(y,x,z)-W(x,y,z)\big)\Dt\vp_{\vep}(y,x,z){\rm d}^3F_{\tau}\nonumber
\\&+&\int_{0<y, z<x} W(x,y,z)
		\Dt\vp_{\vep}(x,y,z){\rm d}^3F_{\tau}\nonumber
\\
		&=&\int_{0< x<y\le z}\chi_{y,z}W(x,y,z)\Dt_{\rm sym}\vp_{\vep}(x,y,z){\rm d}^3F_{\tau}\nonumber\\
&+&
		2\int_{0< x<y<z}\big(W(y,x,z)-W(x,y,z)\big)\Dt\vp_{\vep}(y,x,z){\rm d}^3F_{\tau}\nonumber
\\&+&\int_{0<y, z<x} W(x,y,z)
		\Dt\vp_{\vep}(x,y,z){\rm d}^3F_{\tau}\nonumber \\
&+&F_{\tau}(\{0\})\int_{0<y\le z}\chi_{y,z}W(x,y,z)\Dt_{\rm sym}\vp_{\vep}(0,y,z){\rm d}^2F_{\tau}\nonumber\\
&+&2F_{\tau}(\{0\})\int_{0<y<z}\big(W(y,0,z)-W(0,y,z)\big)\Dt\vp_{\vep}(y,,z){\rm d}^2F_{\tau}\nonumber\\
&:=& I_1(\tau)+I_2(\tau)+I_3(\tau)+I_4(\tau)+I_5(\tau),
 \ees
		where
		\bes&& \Dt_{\rm sym}\vp(x,y,z)=
		 \vp(z+y-x)+\vp(z+x-y)-2\vp(z)\nonumber\\
		&&=(y-x)^2\int_{0}^1\!\!\!\int_{0}^1\vp''
		(z+(s-t)(y-x))
		{\rm d}s{\rm d}t,\quad 0\le x, y\le z.\lb{4.7} \ees
\bes\chi_{y,z} =\left\{\begin{array}{ll} 2
 \,\,\,\,\qquad\quad {\rm if} \quad y< z,\\
1\,\,\,\,\qquad\quad {\rm if} \quad y=z.
\end{array}\right.\nonumber\ees
		Now we are going to prove that
\bes&&\limsup_{\vep \to 0^+}\int_0^t \big(I_1(\tau)+I_2(\tau)+I_3(\tau)\big) d\tau =0 \lb{K0}\\
&&\limsup_{\vep \to 0^+}\int_0^t \big(I_4(\tau)+I_5(\tau)\big) d\tau \le \int_0^t F_{\tau}(\{0\})8^{2+\eta}b_0^2(1+q_1)N^2 {\rm d}\tau \dnumber \lb{II}
\ees
It is easy to deduce that
\beas&&\lim_{\vep \to 0^+}W(x,y,z)\Dt_{\rm sym}\vp_{\vep}(x,y,z)=0 \quad for\quad all \quad 0<x<y\le z, \\
&&\lim_{\vep \to 0^+}(W(y,x,z)-W(x,y,z)\big)\Dt\vp_{\vep}(y,x,z)=0 \quad {\rm for}\,\,\,{\rm all}\,\,\, 0<x<y\le z, \\
&&\lim_{\vep \to 0^+}W(x,y,z)\Dt\vp_{\vep}(x,y,z)=0 \quad {\rm for}\,\,\,{\rm all}\,\,\, 0<y,z< x<y+z.
\eeas
This triggers us to use dominated convergence theorem to prove (\ref{K0}). If we can prove
\bes &&W(x,y,z)\Dt_{\rm sym}\vp_{\vep}(x,y,z) \le  8^{1+\eta}b_0^2 \quad for\quad all \quad 0<x<y\le z,\lb{E1}\\
&&(W(y,x,z)-W(x,y,z)\big)\Dt\vp_{\vep}(y,x,z) \le 8^{1+\eta}q_1 b_0^2 \quad {\rm for}\,\,\,{\rm all}\,\,\, 0<x<y\le z,\dnumber \lb{E2} \\
&&W(x,y,z)\Dt\vp_{\vep}(x,y,z)=0 \quad {\rm for}\,\,\,{\rm all}\,\,\, 0<y,z< x<y+z, \dnumber \lb{E3} \ees
then we can use dominated convergence theorem to prove (\ref{K0}). To prove (\ref{E1}), by the convexity of $\vp_1$ we have $0\le -\vp_1'(x)\le \fr{\vp_1(0)-\vp_1(x)}{x}\le \fr{1}{x}$, thus
		$$\quad \Dt_{\rm sym}\vp_{\vep}(x,y,z)\le \vp_1(\fr{z+x-y}{\vep})-\vp_1(\fr{z}{\vep}) \le -\vp_1'(\fr{z+x-y}{\vep})\fr{y-x}{\vep}\le \fr{y-x}{z+x-y} \quad \forall 0\le x<y\le z $$
		So we have for $0\le x<y\le \fr{z}{2} $, $$W(x,y,z)\Dt_{\rm sym}\vp_{\vep}(x,y,z)\le 4b_0^2\fr{\min\{1,\{8z\}^\eta\}}{\sqrt{y}\sqrt{z}}\fr{y-x}{z-y+x}\le 8^{1+\eta}b_0^2, $$
		and for  $0\le x<y\le z,y>\fr{z}{2} $,
		$$W(x,y,z)\Dt_{\rm sym}\vp_{\vep}(x,y,z)\le 4b_0^2\fr{\min\{1,\{8z\}^\eta\}}{\sqrt{y}\sqrt{z}}\le 8^{1+\eta}b_0^2.   $$
For the term $\big(W(y,x,z)-W(x,y,z)\big)\Dt\vp_{\vep}(y,x,z)$, if  $W(y,x,z)\ge W(x,y,z)$, then  $\big(W(y,x,z)-W(x,y,z)\big)\Dt\vp_{\vep}(y,x,z)\le 0$ for $0\le x<y<z$.	
If $W(y,x,z)\le W(x,y,z)$, and $0\le x<y<\fr{z}{2}$, then
			$$\big(W(y,x,z)-W(x,y,z)\big)\Dt\vp_{\vep}(y,x,z)\le q_1\fr{y}{z}W(y,x,z)\vp_{\vep}(x)\le q_1\fr{y}{z}4b_0^2 \fr{\min\{1,(8z)^\eta\}}{\sqrt{yz}} \le 8^{1+\eta}q_1 b_0^2.$$
			For  $W(y,x,z)\le W(x,y,z)$ and $0\le x<y<z,y\ge \fr{z}{2}$, we have
			$$\big(W(y,x,z)-W(x,y,z)\big)\Dt\vp_{\vep}(y,x,z)\le 4b_0^2 \fr{\min\{1,(8z)^\eta\}}{\sqrt{yz}}\vp_{\vep}(x) \le 8^{1+\eta}b_0^2.$$
So we have proved (\ref{E1}),(\ref{E2}),(\ref{E3}) thus (\ref{K0}) holds. The proof of (\ref{II}) is analogous, in fact we can use the same method to prove that
\beas &&\lim_{\vep \to 0^+}W(0,y,z)\Dt_{\rm sym}\vp_{\vep}(0,y,z)=0 \quad for\quad all \quad 0<y<z,\\
&&\lim_{\vep \to 0^+}(W(y,0,z)-W(0,y,z)\big)\Dt\vp_{\vep}(y,x,z)=0 \quad {\rm for}\,\,\,{\rm all}\,\,\, 0<y<z,\\
 &&W(0,y,z)\Dt_{\rm sym}\vp_{\vep}(0,y,z) \le  8^{1+\eta}b_0^2 \quad for\quad all \quad 0<y\le z,\\
&&\big((W(y,0,z)-W(0,y,z)\big)\Dt\vp_{\vep}(y,0,z) \le 8^{1+\eta}q_1 b_0^2 \quad {\rm for}\,\,\,{\rm all}\,\,\, 0<y\le z. \eeas
The only difference is that we can not prove
\beas &&\lim_{\vep \to 0^+}W(0,y,z)\Dt_{\rm sym}\vp_{\vep}(0,y,z)=0 \quad for\quad all \quad 0<y=z,\\
&&\lim_{\vep \to 0^+}(W(y,0,z)-W(0,y,z)\big)\Dt\vp_{\vep}(y,x,z)=0 \quad {\rm for}\,\,\,{\rm all}\,\,\, 0<y=z,\eeas
		Combining (\ref{J^+}),(\ref{K0}),(\ref{II}), and  taking  sup limits in (\ref{I1}) as $\vep \to 0^+$ we have
		$$F_t(\{0\})\le  F_0(\{0\})+\int_0^t F_{\tau}(\{0\})8^{2+\eta}b_0^2(1+q_1)N^2 {\rm d}\tau.$$
		So by Gronwall inequality we conclude
		$$F_t(\{0\}) \le e^{8^{2+\eta}b_0^2(1+q_1)N^2t} F_0(\{0\}),\quad t\ge 0. $$  $\hfill\Box$

\begin{remark} {\rm The above inequality, i.e. $F_t(\{0\}) \le e^{Ct} F_0(\{0\})$,
 is very special and has an obvious physics meaning: under the assumption about balanced potential,
 if there is no seed of condensation at the origin,
 then there is always no condensation at the origin. However this property does not hold for a set
 away from the origin, i.e. the inequality like  $F_t(\{x\})\le e^{Ct} F_0(\{x\})$ may not hold for $x>0$.
 In the following we only show this phenomenon for $x$ belonging to a set of positive intergers. The proof for other $x\in (0,\infty)$ is essentially the same.}
\par
\end{remark}

\noindent{\bf Example (propagation of singularity away from the origin).}

Let $F_0\in {\cal B}_{1}^{+}({\mR}_{\ge 0})$ satisfy $F_0(\{1\})>0, F_0(\{2\})>0$. Let $F_t\in {\cal B}_{1}^{+}({\mR}_{\ge 0})$ with the initial datum
$F_0$ be a conservative measure-valued solution of Eq.(\ref{Equation1}) where the collision kernel
$B$ together with $F_0$ satisfies one of the following two conditions:
\par

	{\rm(a)} $B({\bf {\bf v-v}_*},\omega)$ satisfies  Assumption \ref{assp} with $\inf\limits_{r\ge R}\widehat{\phi}(r) >0$ for all $R>0$, and $M_{-1/2}(F_0)<\infty$;
\par

	{\rm(b)} $B({\bf{\bf v-v}_*},\omega)=\frac{1}{(4\pi)^2}|({\bf v-v}_*)\cdot\omega|$ (the hard sphere model) and $M_{-1/2}(F_0)\le \fr{1}{320}[N(F_0)E(F_0)]^{1/4}$.  \par
\noindent Then $F_t(\{n\})>0 $ for all $n\in {\mN}$ and all $t> 0$.
\par
\noindent\begin{proof} In the proof we will use some notations and results in Section 3 and Section 4.
First of all we note that each of the conditions (a), (b) implies that the solution $F_t$ is unque (
see Theorem \ref{theorem1.11} and Theorem 3.2 of \cite{Lu2014}), and this allows us to use approximate solutions.
 \par
Part {\rm (a)}:  Denote $ a(r)=\inf\limits_{l\ge r}\widehat{\phi}(l))$, $u=F_0(\{1\}),v=F_0(\{2\}), H_0=F_0-u\delta(\cdot-1)-v\delta(\cdot-2)$, where $\delta(\cdot)$ is the Dirac measure concentrated at $x=0$. For any $2\le k\in {\mN}$, let
	$$ f_{0,k}(x) = \fr{ku}{2}1_{[1-\fr{1}{k},1+\fr{1}{k}]}(x)+\fr{kv}{2\sqrt{2}}1_{[2-\fr{1}{k},2+\fr{1}{k}]}(x)+\tilde{f}_{0,k}(x),      \quad \quad \quad  x\in (0,\infty) $$
	where $\tilde{f}_{0,k}(\cdot)\sqrt{\cdot}$ converges to $H_0$ weakly, i.e
	$$\lim_{k\to \infty}\int_{{\mathbb R}_{\ge 0}}\vp \tilde{f}_{0,k}(x)\sqrt{x} {\rm d}x=\int_{{\mathbb R}_{\ge 0}}\vp  {\rm d}H_0(x) \quad \forall \vp\in C_b({\mathbb R}_{\ge 0}). $$
	We can choose $\tilde{f}_{0,k}$
	appropriately such that
\bes&& \fr{1}{2}N(F_0)\le N(f_{0,k})\le 2 N(F_0),\quad  \fr{1}{2}E(F_0)\le E(f_{0,k})\le 2 E(F_0),\nonumber\\
&&  \fr{1}{2}M_{-1/2}(F_0) \le M_{-1/2}(f_{0,k})\le 2M_{-1/2}(F_0)  \qquad \forall k\ge 2. \lb{f}\ees
It is obvious that  $f_{0,k}(\cdot)\sqrt{\cdot}$ converges weakly to $F_0$. Using Lemma 2.2 and Lemma 2.3 in \cite{Lu2014}, $W(x,y,z)\le 4b_0^2 W_H(x,y,z)$, Theorem \ref{theorem1.3}, and Theorem \ref{theorem1.10},
we know there  exist  unique conservative mild solutions $f_k$ on ${\mR}_{\ge 0}\times [0,\infty)$ with initial data $f_{0,k} $ and satisfies for any $T\in [0,\infty)$,
\beas&& \sup_{\tau \in [0,T], x \le 5, k\ge 2 }L(f_k)(x,\tau)\\
&&\le \sup_{\tau \in [0,T], x \le 5,k \ge 2 }4b_0^2(\sqrt{x} N(f_k)+M_{1/2}(f_k)(\tau)+2[M_{-1/2}(f_k)]^2(\tau)):=C_T<\infty.\eeas
	So by Proposition \ref{proposition 3-3}  we get
	\beas f_k(x,t)&=&f_{0,k}(x) e^{-\int_{0}^{t}L(f_k)(x,\tau){\rm d}\tau}
	+\int_{0}^{t}Q^{+}(f_k)(x,\tau)e^{-\int_{\tau}^{t}L(f_k)(x,s){\rm d}s}\\
&&\ge f_{0,k}(x) e^{-\int_{0}^{t}L(f_k)(x,\tau){\rm d}\tau}\ge f_{0,k}(x) e^{-\int_{0}^{t}C_T{\rm d}\tau}=f_{0,k}(x) e^{-C_Tt}  \eeas
	for all $x\in [0,4]$, $k\ge 2$ and $t \in [0,T]$.
		By (\ref{W2}), we calculate
	\beas &&W(x,y,z)\sqrt{y}\sqrt{z}=\fr{1}{4\pi\sqrt{x}}
	\int_{|\sqrt{x}-\sqrt{y}|\vee |\sqrt{x_*}-\sqrt{z}|}
	^{(\sqrt{x}+\sqrt{y})\wedge(\sqrt{x_*}+\sqrt{z})}{\rm d}s
	\int_{0}^{2\pi}\Phi(\sqrt{2}s, \sqrt{2} Y_*){\rm d}\theta\\
	&&\ge \fr{1}{4\pi\sqrt{x}}
	\int_{|\sqrt{x}-\sqrt{y}|\vee |\sqrt{x_*}-\sqrt{z}|}
	^{(\sqrt{x}+\sqrt{y})\wedge(\sqrt{x_*}+\sqrt{z})}{\rm d}s
	\int_{0}^{2\pi}a^2(\fr{1}{8}){\rm d}\theta=\fr{\sqrt{x_*}}{2\sqrt{x}}a^2(\fr{1}{8})\ge \fr{1}{2\sqrt{13}}a^2(\fr{1}{8}):=c_1 \eeas
for all $x\in [\fr{11}{4},\fr{13}{4}]$,$y\in [\fr{7}{4},\fr{9}{4}]$,$z\in [\fr{7}{4},\fr{9}{4}]$.
So for all $x\in [3-\fr{1}{2k},3+\fr{1}{2k}]$
\beas&&Q^{+}(f_k)(x,\tau)\ge \int_{{\mR}_{\ge 0}^2}W(x,y,z)f_k(y,\tau)f_k(z,\tau)f_k(x_*,\tau)\sqrt{y}\sqrt{z}{\rm d}y{\rm
		d}z\\
&& \ge \int_{{y,z\in [2-\fr{1}{4k}]},2+\fr{1}{4k}]}c_1f_k(y,\tau)f_k(z,\tau)f_k(x_*,\tau){\rm d}y{\rm
		d}z\\
&&\ge \int_{{y,z\in [2-\fr{1}{4k}]},2+\fr{1}{4k}]}e^{-3C_T
		\tau}\fr{c_1k^3uv^2}{16}{\rm d}y{\rm
		d}z=c_1e^{-3C_T
		\tau}\fr{c_1kuv^2}{64} \quad \tau \in [0,T].\eeas
	Thus  for all $x\in [3-\fr{1}{2k},3+\fr{1}{2k}]$, $t \in [0,T]$, we have
	\beas f_k(x,t)&=&f_{0,k}(x) e^{-\int_{0}^{t}L(f_k)(x,\tau){\rm d}\tau}
	+\int_{0}^{t}Q^{+}(f_k)(x,\tau)e^{-\int_{\tau}^{t}L(f_k)(x,s){\rm d}s}\\
&\ge& \int_{0}^{t}e^{-3C_T
		\tau}\fr{c_1kuv^2}{64}e^{-C_T
		(t-\tau)} {\rm d}s =\fr{e^{-C_Tt}-e^{-3C_Tt}}{C_T}\fr{c_1kuv^2}{128}.\eeas
This leads to
	\beas F_{k,t}([3-\fr{1}{2k},3+\fr{1}{2k}])&=&\int_{3-\fr{1}{2k}}^{3+\fr{1}{2k}} f_k(x,t)\sqrt{x}dx\ge\int_{3-\fr{1}{2k}}^{3+\fr{1}{2k}}\fr{e^{-C_Tt}-e^{-3C_Tt}}{C_T}\fr{c_1kuv}{128} \sqrt{2}dx \\
	&=&\fr{e^{-C_Tt}-e^{-3C_Tt}}{C_T}\fr{c_1uv^2}{64\sqrt{2}},  \quad \quad   \tau \in [0,T].\eeas
 From Theorem \ref{weak stability} and Theorem \ref{theorem1.11} (uniqueness), we conclude that the unique conservative measure-valued solution $F_t$ satisfies
	$$F_{t}([3-\fr{1}{2k},3+\fr{1}{2k}])\ge\fr{e^{-C_Tt}-e^{-3C_Tt}}{C_T}\fr{c_1uv^2}{64\sqrt{2}} \quad \quad t \in [0,T]. $$
	Since $k$ can be arbitrarily large, it follows that  $$F_t(\{3\})\ge\fr{e^{-C_Tt}-e^{-3C_Tt}}{C_T}\fr{c_1uv^2}{64\sqrt{2}}>0  \quad \quad t \in [0,T].$$
	So far we know $F_t(\{1\})>0,F_t(\{2\})>0,F_t(\{3\})>0$ for all $t>0$. In particular for any $0<T<\infty$,we know $F_{\fr{T}{2}}(\{3\})>0 $,$F_{\fr{T}{2}}(\{2\})>0 $, $M_{-1/2}(F_{\fr{T}{2}})<\infty$. So we can use $F_{\fr{T}{2}}$ as initial datum and use the same method to get $F_{t}(\{4\})>0 $ for all $\fr{T}{2}<t\le T$. And we can use $F_{\fr{3T}{4}} $ as initial datum to get $F_{t}(\{5\})>0 $ for all $\fr{3T}{4}<t\le T$. By induction we can get $F_{t}(\{n\})>0 $ for all $T-\fr{T}{2^{n-3}}<t\le T$. In particular we can choose $t=T$, then $F_{T}(\{n\})>0.$Since $T>0$ is arbitrary, we get the conclusion.\par
Part{\rm (b)}: We use notations and choose   $f_{0,k}$ just the same as in (I). By (\ref{f}) we have $\|f_{0,k}\|_{L^1}\le \fr{1}{80}[N(f_{0,k})E(f_{0,k})]^{\fr{1}{4}}$. Using Lemma 2.2,Lemma 2.3 and Theorem 4.1 in \cite{Lu2014}, we know there  exist  unique conservative mild solutions $f_k$ on ${\mR}_{\ge 0}\times [0,\infty)$ with initial data $f_{0,k} $ and satisfies
	\beas&& \sup_{\tau \in [0,\infty], x \le 5, k\ge 2 }L(f_k)(x,\tau)\\
	&&\le \sup_{\tau \in [0,\infty], x \le 5,k \ge 2 }\sqrt{x} N(f_k)+M_{1/2}(f_k)(\tau)+2[M_{-1/2}(f_k)]^2(\tau):=C<\infty.\eeas
	In a way similar to the proof of (I) and using $W_H(x,y,z)\sqrt{yz}=\fr{\min\{\sqrt{x},\sqrt{x_*},\sqrt{y},\sqrt{z}\}}{\sqrt{x}}$, we get
$$ F_{k,t}([3-\fr{1}{2k},3+\fr{1}{2k}])\ge \fr{e^{-Ct}-e^{-3Ct}}{C}\fr{uv^2}{192\sqrt{2}}.$$
	By Theorem \ref{weak stability} and Theorem 3.2 in \cite{Lu2014} (uniqueness), we conclude that the conservative measure-valued solution $F_t$ satisfies
	$$F_{t}([3-\fr{1}{2k},3+\fr{1}{2k}])\ge\fr{e^{-Ct}-e^{-3Ct}}{C}\fr{uv^2}{192\sqrt{2}}. $$
	Since $k$ can be arbitrarily large, we obatain  $$F_t(\{3\})\ge\fr{e^{-Ct}-e^{-3Ct}}{C}\fr{uv^2}{192\sqrt{2}}>0.$$
	Now we have proved $F_t(\{1\})>0,F_t(\{2\})>0,F_t(\{3\})>0$ for all $t>0$. The rest of the parts of prove is just the same as in Part (a).
\end{proof}
\begin{remark}{\rm  This Example also tells us that for many initial data $F_0$, there is no hope for $F_t$
(with $t>0$) to be decomposed as
${\rm d}F_t(x)=f(x,t)\sqrt{x}{\rm d}x+n(t)\dt(x){\rm d}x$
where $0\le f(\cdot, t)\in L^1({\mR}_{\ge 0},\sqrt{x}{\rm d}x)$, $n(t)\ge 0$ and
$\dt(\cdot)$ is the Dirac delta function concentrated at $x=0$.}
\end{remark}

\par
	\begin{theorem}\lb{theorem1.3} Let $F_0\in {\mathcal B}_1^{+}({\mathbb R}_{\ge 0})$  with mass $N=N(F_0)$ and energy $E=E(F_0)$. Given any $0<p\le \fr{1}{2}$, suppose $B(v-v_*,w)$ satisfy the Assumption \ref{assp} with $\eta\ge 1+p$  and the initial $F_0$ satisfy $M_{-p}(F_0)<\infty$. Let
	$F_t$ be a conservative  measure-valued isotropic solution $F_t$ of
	Eq.(\ref{Equation1}) on $[0,\infty)$ with the initial datum $F_0$.
	Then $M_{-p}(F_t)<\infty$ for all $t>0$. More precisely we have
	\begin{align} M_{-p}(F_t)\le  (at+M_{-p}(F_0)) e^{bt}\quad \forall t\ge 0 \dnumber \lb{nmoment}\end{align}
	where
	$a=8^2b_0^2 N^{\fr{3}{2}+p}E^{\fr{1}{2}-p}+8^{2+\eta}b_0^2N^3, \,b=8^{3+\eta}b_0^2N^2(1+q_1)$.
	
\end{theorem}
\begin{proof}Denote $\varphi_{\vep,p}(x)=\fr{1}{(\vep+x)^p}$ and $ M_{-p}^{\vep}(F_t)=\int_{{\mR}_{\ge 0}} \varphi_{\vep,p}(x)dF_t(x)$. To prove (\ref{nmoment}), first we prove the following differential inequality of $ M_{-p}^{\vep}(F_t)$:
	$$\fr{d}{dt}M_{-p}^{\vep}(F_t)\le a+bM_{-p}^{\vep}(F_t).$$
	Recalling that $W_H(x,y,z)=\fr{1}{\sqrt{xyz}}\min\{\sqrt{x},\sqrt{y},\sqrt{z},\sqrt{x_*}\} $ and  $W(x,y,z)\le 4b_0^2 W_H(x,y,z) $, for $ 0<y \le z$
	we have
	\beas &&{\cal J}[\vp_{\vep,p}](y,z)\le \frac{1}{2}
	\int_{0}^{y+z}W(x,y,z)(\vp_{\vep,p}(x)+\vp_{\vep,p}(y+z-x))
	\sqrt{x}{\rm d}x \nonumber \\
	&& \le 2b_0^2\int_{0}^{y+z}W_H(x,y,z)(\vp_{\vep,p}(x)+\vp_{\vep,p}(y+z-x))
	\sqrt{x}{\rm d}x=4b_0^2\int_{0}^{y+z}W_H(x,y,z)\vp_{\vep,p}(x)
	\sqrt{x}{\rm d}x  \lb{JJ1}\eeas
and
	\beas &&\int_{0}^{y+z}W_H(x,y,z)\vp_{\vep,p}(x)
	\sqrt{x}{\rm d}x \le\int_{0}^{y} \fr{\sqrt{x}}{\sqrt{yz}} \fr{1}{x^p} {\rm d}x+\int_{y}^{z} \fr{1}{\sqrt{z}} \fr{1}{x^p} {\rm d}x+\int_{z}^{y+z} \fr{\sqrt{x_*}}{\sqrt{yz}} \fr{1}{x^p} {\rm d}x    \\
	&&\le \fr{1}{\fr{3}{2}-p} \fr{y^{1-p}}{\sqrt{z}}+\fr{1}{1-p} \fr{z^{1-p}-y^{1-p}}{\sqrt{z}}+\fr{1}{1-p} \fr{(y+z)^{1-p}-z^{1-p}}{\sqrt{z}} \le 2 \fr{(y+z)^{1-p}}{\sqrt{z}} \le 4 z^{\fr{1}{2}-p} .\eeas
	 By symmetry of $y,z$ we obtain
	$$\int_{0}^{y+z}W_H(x,y,z)\vp_{\vep,p}(x)\sqrt{x}{\rm d}x \le 4(y^{\fr{1}{2}-p}+z^{\fr{1}{2}-p}) \quad {\rm for}\,\,\,{\rm all}\,\,\, 0\le y,z.$$
	Thus we can use H{\"o}lder inequality to get
	\bes \int_{\mR^2 \ge 0}{\cal J}[\vp_{\vep,p}](y,z){\rm d}^2F_t\le\int_{\mR^2 \ge 0}16(y^{\fr{1}{2}-p}+z^{\fr{1}{2}-p})b_0^2 {\rm d}^2F_t \le  32b_0^2 N^{\fr{3}{2}+p}E^{\fr{1}{2}-p}. \lb{JE}\ees
	For the cubic term,we use the following decomposition again
	\bes \int_{{\mathbb R}_{\ge 0}^3}{\cal K}[\vp_{\vep,p}]{\rm d}^3F_t
	&=&\left(2\int_{0\le x<y<z}+2\int_{0\le y<x<z}+
	\int_{0\le x<y=z}+\int_{0\le y, z<x}
	\right)W(x,y,z)\Dt\vp_{\vep,p}(x,y,z){\rm d}^3F_t\nonumber \\
&=&\int_{0\le x<y\le z}\chi_{y,z}W(x,y,z)\Dt_{\rm sym}\vp_{\vep,p}(x,y,z){\rm d}^3F_t
	\nonumber\\
	&+&
	2\int_{0\le x<y<z}\big(W(y,x,z)-W(x,y,z)\big)\Dt\vp_{\vep,p}(y,x,z){\rm d}^3F_t.\nonumber \\
&+&\int_{0<y, z<x} W(x,y,z)
	\Dt\vp_{\vep,p}(x,y,z){\rm d}^3F_t. \lb{de} \ees
	Using Lemma \ref{lemma0},(\ref{4.7}) and the fact that $\eta\ge \fr{3}{2}$, $0<p\le \fr{1}{2}$, $\varphi_{\vep,p}$ is convex and decreasing  we can get the following estimates:
	\bes &&\int_{0\le x<y\le z}\chi_{y,z}W(x,y,z)\Dt_{\rm sym}\vp_{\vep,p}(x,y,z){\rm d}^3F_t \le 2\int_{0\le x<y\le z}W(x,y,z)\Dt_{{\rm sym}}\vp_{\vep,p}(x,y,z){\rm d}^3F_t  \nonumber \\
	&&=2\int_{0\le x<y<\fr{z}{2}}W(x,y,z)\Dt_{{\rm sym}}\vp_{\vep,p}(x,y,z){\rm d}^3F_t+2\int_{0\le x<y\le z,y\ge \fr{z}{2}}W(x,y,z)\Dt_{{\rm sym}}\vp_{\vep,p,}(x,y,z){\rm d}^3F_t  \nonumber \\
	&&\le \int_{0\le x<y<\fr{z}{2}}8b_0^2 \fr{\min\{1,(8z)^\eta\}}{\sqrt{yz}}p(p+1) \fr{(y-x)^2}{(\vep+z+x-y)^{p+2}}{\rm d}^3F_t \nonumber \\ &&+\int_{0\le x<y\le z,y\ge \fr{z}{2}}8b_0^2 \fr{\min\{1,(8z)^\eta\}}{\sqrt{yz}}\vp_{\vep,p}(z+x-y){\rm d}^3F_t \nonumber \\
	&&\le \int_{0\le x<y<\fr{z}{2}}8b_0^2p(p+1) \fr{\min\{1,(8z)^\eta\}}{\sqrt{yz}} \fr{y^2}{(\fr{z}{2})^{p+2}}{\rm d}^3F_t+\int_{0\le x<y\le z,y\ge \fr{z}{2}}8\sqrt{2}b_0^2 \fr{\min\{1,(8z)^\eta\}}{z}\vp_{\vep,p}(x){\rm d}^3F_t  \nonumber \\
	&&\le 8^{2+\eta}b_0^2 N^3+ 8^{2+\eta}b_0^2N^2 M_{-p}^{\vep}(F_t), \lb{KK1}\ees
	\bes && 2\int_{0\le x<y<z}\big(W(y,x,z)-W(x,y,z)\big)\Dt\vp_{\vep,p}(y,x,z){\rm d}^3F_t \nonumber\\
	&&\le 2\int_{0\le x<y<z,W(x,y,z)\ge W(y,x,z) }\big(W(x,y,z)-W(y,x,z)\big)\vp_{\vep,p}(x){\rm d}^3F_t \nonumber \\
	&& \le 2\int_{0\le x<y<\fr{z}{2}}\big(W(x,y,z)-W(y,x,z)\big)\vp_{\vep,p}(x){\rm d}^3F_t	\nonumber \\&&+2\int_{0\le x<y<z,y\ge \fr{z}{2}, }\big(W(x,y,z)-W(y,x,z)\big)\vp_{\vep,p}(x){\rm d}^3F_t  \nonumber\\
	&&\le \int_{0\le x<y<\fr{z}{2}}2q_1\fr{y}{z}W(y,x,z)\vp_{\vep,p}(x){\rm d}^3F_t+\int_{0\le x<y<z,y\ge \fr{z}{2}}2W(x,y,z)\vp_{\vep,p}(x){\rm d}^3F_t \nonumber\\
	&&\le \int_{0\le x<y<\fr{z}{2}}q_1\fr{y}{z}8b_0^2 \fr{\min\{1,(8z)^\eta\}}{\sqrt{yz}}\vp_{\vep,p}(x){\rm d}^3F_t+\int_{0\le x<y<z,y\ge \fr{z}{2}} 8b_0^2\fr{\min\{1,(8z)^\eta\}}{\sqrt{yz}}\vp_{\vep,p}(x){\rm d}^3F_t\nonumber\\
	&&\le 8^{1+\eta}b_0^2N^2 q_1M_{-p}^{\vep}(F_t)+8^{2+\eta}b_0^2N^2 M_{-p}^{\vep}(F_t). \lb{KK2}\ees

	Since  $0<p\le \fr{1}{2}$, we have that   $\sqrt{t}\vp_{\vep,p}(t)$ is non-decreasing, thus $\sqrt{x_*}\fr{\vp_{\vep,p}(x_*)}{\sqrt{y}}\le \vp_{\vep,p}(y) $ for all $0<y\le z<x<y+z$. Using this inequality we can get the following estimate:
	\bes I_4&=&\int_{0<y, z<x} W(x,y,z)
	\Dt\vp(x,y,z){\rm d}^3F
	\le 8b_0^2\int_{0<y\le z<x<y+z} \fr{\min\{1,\{8x\}^\eta\}}{\sqrt{x}\sqrt{y}\sqrt{z}}\sqrt{x_*}
	\Dt\vp_{\vep,p}(x,y,z){\rm d}^3F \nonumber \\
	&\le& 8b_0^2\int_{0<y\le z<x<y+z} \fr{\min\{1,\{8x\}^\eta\}}{\sqrt{x}\sqrt{y}\sqrt{z}}\sqrt{x_*}
	\vp_{\vep,p}(x_*){\rm d}^3F  \le 8b_0^2\int_{0<y\le z<x<y+z} \fr{\min\{1,\{8x\}^\eta\}}{\sqrt{x}\sqrt{z}}
	\vp_{\vep,p}(y){\rm d}^3F\nonumber \\
	&\le& 8^{2+\eta}b_0^2N^2M_{-p}^{\vep}(F_t) \lb{KK3} \ees
	Combining  (\ref{JE}),(\ref{de}), (\ref{KK1}), (\ref{KK2}), (\ref{KK3}) , we prove the following inequality:
	$$\fr{d}{dt}M_{-p}^{\vep}(F_t)\le 8^2b_0^2 N^{\fr{3}{2}+p}E^{\fr{1}{2}-p}+8^{2+\eta}b_0^2 N^3+8^{3+\eta}b_0^2N^2(1+q_1)M_{-p}^{\vep}(F_t)=a+bM_{-p}^{\vep}(F_t).$$
	Solving this differential inequality we obtain
	$$ M_{-p}^{\vep}(F_t)\le (8^2b_0^2 N^{\fr{3}{2}+p}E^{\fr{1}{2}-p}+8^{2+\eta}b_0^2 N^3)te^{8^{3+\eta}b_0^2N^2(1+q_1)t}+e^{8^{3+\eta}b_0^2N^2(1+q_1)t}M_{-p}^{\vep}(F_0).$$
	Let $\epsilon \to 0^+$ and using the monotone convergence theorem, the above inequality yields
	$$M_{-p}(F_t)\le  (8^2b_0^2 N^{\fr{3}{2}+p}E^{\fr{1}{2}-p}+8^{2+\eta}b_0^2 N^3)te^{8^{3+\eta}b_0^2N^2(1+q_1)t}+e^{8^{3+\eta}b_0^2N^2(1+q_1)t}M_{-p}(F_0)<\infty \quad \forall t\ge 0, $$
	which is the desired result.
\end{proof}
\begin{center}\section { Moment Production and Weak Convergence }\end{center}
To prove moment production and positive lower bound of entropy, as the same in \cite{Cai-Lu}, we inroduce the following definition of a class of approximate solutions:
\begin{definition}\label{definition3.1} Let $B({\bf {\bf v-v}_*},\omega)$ be given by
	(\ref{kernel}), (\ref{kernel2}).
	We say that $\{B_K({\bf {\bf v-v}_*},\omega)\}_{K\in {\mN}}$ is a sequence of approximation of $B({\bf {\bf v-v}_*},\omega)$
	if
	$B_K({\bf {\bf v-v}_*},\omega)$ are such Borel measurable functions on ${\bRS}$ that they are functions of
	$(| {\bf v}-{\bf v}'|, |{\bf v}-{\bf v}_*'|)$ only and satisfy
	$$B_K({\bf v-v}_*,\omega)\ge 0, \quad \lim\limits_{K\to\infty}
	B_K({\bf v-v}_*,\omega)=
	B({\bf v-v}_*,\omega)$$ for a.e $({\bf v-v}_*,\omega)
	\in {\bR}\times {\bS}$.
	Let $Q_K(f)$ be the collision integral operators corresponding to the approximate kernels $B_K$, i.e.
	\be Q_K(f)({\bf v})=
	\int_{{\bRS}}B_K({\bf
		{\bf v-v}_*},\og)\big(f'f_*'(1+f+f_*)-ff_*(1+f'+f_*')\big) {\rm d}\omega{\rm
		d}{\bf v_*}.\label{3.Equation}\ee
	Given any $K\in {\mN}$ and $0\le f^K_0\in L^1_2({\bR})$. We
	say that $f^K=f^K({\bf v},t)$ is a conservative approximate solution of Eq.(\ref{Equation1}) on ${\bR}\times [0,\infty)$ corresponding to the approximate kernel $B_K$  with the initial datum $f^K_0$ if
	$({\bf v},t)\mapsto f^K({\bf v},t)$ is a nonnegative Lebesgue measurable function on ${\bR}\times [0,\infty)$ satisfying
	
	{\rm(i)}
	$\sup_{t\ge 0}\|f^K(t)\|_{L^1_2}<\infty$ (here and below $f^K(t):=f^K(\cdot,t)$) and
	\be\int_{0}^{T}{\rm d}t
	\int_{{\bRRS}}B_K({\bf
		{\bf v-v}_*},\og)(f^K)'(f^K)_*'\big(1+f^K+f^K_*\big)\sqrt{1+|{\bf v}|^2+|{\bf v}_*|^2}\,{\rm d}\omega{\rm
		d}{\bf v}{\rm d}{\bf v_*}<\infty\lb{3.new1}\ee
	for all $0<T<\infty$.
	
	{\rm(ii)}  There is a null set
	$Z\subset {\bR}$ which is independent of\, $t$ such that
	\be
	f^K({\bf v},t)=f^K_0({\bf v})+\int_{0}^{t}Q_K(f^K)({\bf v},\tau){\rm d}\tau\quad \forall\, t\in[0,\infty),\,\,\forall\,{\bf v}\in {\bR}\setminus Z.\lb{3.new2}\ee

	{\rm(iii)}  $f^K$ conserves the mass, momentum, and energy, and satisfies the entropy
	equality, i.e.
	\bes &&
	\int_{{\bR}}(1,{\bf v}, |{\bf v}|^2/2) f^K({\bf v},t){\rm d}{\bf v}=
	\int_{{\bR}}(1,{\bf v}, |{\bf v}|^2/2) f^K_0({\bf v}){\rm d}{\bf v}\qquad \forall\, t\ge 0
	\label{energyconser} \\
	&& S(f^K(t))=S(f^K_0)+\int_{0}^{t}D_K(f^K(\tau)){\rm d}\tau\qquad\forall\, t\ge 0.
	\dnumber\label{eelity}\ees
	Here $Q_K(f^K)({\bf v},t)=Q_K(f^K(\cdot,t))({\bf v})$, $D_K(f)$ is the entropy dissipation corresponding to the approximate kernel $B_K({\bf
		{\bf v-v}_*},\og)$, i.e.
	\be D_K(f)=\fr{1}{4}
	\int_{{\bRRS}}B_K({\bf
		{\bf v-v}_*},\og)\Pi(f)\Gm(g'g_*', gg_*) {\rm d}\omega{\rm
		d}{\bf v_*}{\rm
		d}{\bf v}\label{3.2}\ee
	Where \be \Gm(a,b)=
\left\{\begin{array}{ll}
	\displaystyle (a-b)\log\big(\fr{a}{b}\big)\,\qquad {\rm if}
	\quad a>0, b>0\\
	\displaystyle
	\,\infty\qquad \quad\qquad \qquad \,\,  {\rm if}\quad
	a>0=b\,\,\,{\rm or}\,\,\, a=0<b\\
	\displaystyle
	\,0\qquad \qquad \qquad \quad\,\,\,\,\,  {\rm if}\quad a=b=0
\end{array}\right.\lb{2.Gamma}\ee
\be\Pi(f)=(1+f)(1+f_*)(1+f')(1+f_*'),\quad g=\fr{f}{1+f},\quad \lb{2.3}\ee
	
	\noindent If a conservative approximate  solution $f^K$ is isotropic, i.e. if
	$f^K({\bf v},t)\equiv f^K(|{\bf v}|^2/2,t)$,
	then $f^K$ is called a conservative isotropic approximate
	solution of Eq.(\ref{Equation1}).In this case,if we define $h^K(x)= f^K(|{\bf v}|^2/2,t)$ for $x=|{\bf v}|^2/2$,then $h^K$ is a mild solution in the sense of Definition (\ref{definition 1-1}).
\end{definition}
A suitbale class of $B_K$ that was often be used is
\be B_K({\bf {\bf v-v}_*},\omega)=\min\big\{B({\bf {\bf v-v}_*},\omega),\, K|{\bf v}-{\bf v}'|^2
|{\bf v}-{\bf v}_*'|\big\},\quad K\ge 1. \label{3.1}\ee

An important Theorem will often be used below is Theorem 1 in \cite{Lu2004} (weak stability). Notice that the condition $\int^1_0 B(V,\tau){\rm d}\tau >0$ for all $V>0$ was not used in the proof of Theorem 1 in \cite{Lu2004} (weak stability), so using Appendix of \cite{Cai-Lu} (Equivalence of Solutions) we would like to rephrase that Theorem into the following form.
\begin{theorem}\lb{weak stability}\cite{Lu2004}({\bf Weak Stability}). Let $B,B_K$ be collision kernels satisfying the conditions \par
	{\rm(i)}$B(\cdot,\cdot) \in C({\mathbb R}_{\ge 0}\times [0,1])$,
\par
	{\rm(ii)} $\sup\limits_{V\ge 0,\tau \in [0,1]} \fr{B(V,\tau)}{1+V}<\infty, \quad \sup\limits_{V\ge 0} \fr{B(V,\tau)}{1+V}\to 0 \quad as \quad \tau \to 0^+.$
\par
\noindent and either $B_K\equiv B $ ($\forall K \ge 1)$ or $B_K$ be the cutoff of $B$ given by (\ref{3.1}). Let $F_0,F^K_0\in {\cal B}_{1}^{+}({\mR}_{\ge 0})$ satisfying
$$\sup\limits_{K \ge 1} \int_{{\mR}_{\ge 0}} (1+x){\rm d}F^K_0(x)<\infty $$
and
$$\lim\limits_{n\to \infty}\int_{{\mR}_{\ge 0}}\varphi(x){\rm d}F^K_0(x)=\int_{{\mR}_{\ge 0}}\varphi(x){\rm d}F_0(x) \quad \quad \forall \varphi \in C_b({\mR}_{\ge 0}) $$
Let $F^n_t$ be conservative distributional solutions of Eq.(\ref{Equation1}) with kernel $B_n$ and initial datum $F^n_0$. Then there exist a subsequence $\{F^{K_j}_t\}_{j=1}^{\infty}$ and a conservative  distributional $F_t$ of Eq.(\ref{Equation1}) with the kernel $B$ and initial datum $F_0$ such that
$$\lim\limits_{n\to \infty}\int_{{\mR}_{\ge 0}}\varphi(x){\rm d}F^{K_j}_t(x)=\int_{{\mR}_{\ge 0}}\varphi(x){\rm d}F_t(x)\quad  \forall \,\, t\ge 0,\quad  \varphi \in C_b({\mR}_{\ge 0}) $$
(therefore $F_t$ conserves the mass ) and
$$\int_{{\mR}_{\ge 0}}x{\rm d}F_t(x)=\liminf\limits_{j \to\infty}\int_{{\mR}_{\ge 0}}x{\rm d}F^{K_j}_t(x).$$
Furthermore if
$$\lim\limits_{K \to \infty} \int_{{\mR}_{\ge 0}} x{\rm d}F^K_0(x)= \int_{{\mR}_{\ge 0}} x{\rm d}F_0(x), $$
then the solution $F_t$ also conserves the energy:
$$\int_{{\mR}_{\ge 0}} x{\rm d}F_t(x)=\int_{{\mR}_{\ge 0}} x{\rm d}F_0(x). $$
\end{theorem}
\par

\vskip2mm
\begin{proposition}\lb{Mproduct} Let the collision kernel $B({\bf {\bf v-v}_*},\omega)$ be given by
	(\ref{kernel}),(\ref{kernel2}),(\ref{Phi}) where the Fourier transform
$r\mapsto \widehat{\phi}(r)$  is  non-negative on $[0,\infty)$ and satifies
\be\lb{3.10} a_0 r^{-\beta}1_{[R, \infty)}(r)\le \widehat{\phi}(r)\le b_1\qquad \forall\, r>0 \ee
 for some constants $a_0>0, 0\le \beta<1/2, 0<R<\infty$. Given any $N>0, E>0$.
	
	{\rm(I)} Let $B_K({\bf {\bf v-v}_*},\omega)$
	be given by (\ref{3.1}) ,
	let $\{f_0^K=f^K_0(|{\bf v}|^2/2)\}_{K\in{\mN}} $ be any sequence of nonnegative
	isotropic functions in $L^1_2({\bR})$  satisfying
	\be \int_{{\bR}}(1, |{\bf v}|^2/2)f^K_0(|{\bf v}|^2/2){\rm d}{\bf v}= 4\pi\sqrt{2}(N, E)\qquad
	\forall\, K\in{\mN}.\lb{entropy.3}\ee
	Then for every $K\in{\mN}$, there exist a unique conservative isotropic approximate  solution
	$f^K=f^K(|{\bf v}|^2/2,t)$ of
	Eq.(\ref{Equation1}) on ${\bR}\times [0,\infty)$  corresponding to the approximate kernel
	$B_K$ such that $f^K|_{t=0}=f^K_0$, and
	it holds the moment production:
	\be \sup_{K\in{\mN}}\|f^K(t)\|_{L^1_s}\le C_s(1+1/t)^{s-2}\quad \forall\, t>0,\,\,\forall\, s>2 \label{mprod}\ee
	where the constant $0<C_s<\infty$ depends only on $N, E,s$,$a_0$,$b_1$,$R$ and $\beta$.

	{\rm(II)}
	Let $F_0\in {\mathcal B}_1^{+}({\mathbb R}_{\ge 0})$ satisfy
	$N(F_0)=N, E(F_0)=E$.  Then there exists a conservative  measure-valued isotropic solution $F_t$ of Eq.(\ref{Equation1}) on $[0,\infty)$ with the initial datum $F_0$, such that
	\be M_{p}(F_t)\le C_p(1+1/t)^{2(p-1)}\quad \forall\,t>0,\,\,\forall\, p>1\lb{Mprod}\ee
	where the constant $0<C_p<\infty$ depends only on
	$N, E,p$,,$a_0$,$b_1$ ,$R$ and $\beta$.
\end{proposition}
\noindent\begin{proof}We only need to prove part (I), part (II) is only an application of part (I) and .
	The existence
	of the
	conservative isotropic approximate  solutions
	$f^K$ has been  proven in
	Theorem 3 of \cite{Lu2000}. So we only need to prove moment production and uniqueness. Without loss of generality, we can assume $R\ge 1$.
	
	For notation convenience we denote (with $K$ fixed)
	$$f^K(|{\bf v}|^2/2,t):=f^K(|{\bf v}|^2/2,t).$$
	To prove (\ref{mprod}) ,we prove it holds for  the case that $\|f_0\|_{L^1_s}<\infty$ for all $s>2$, then we can get the general case by Theorem \ref{weak stability}.
	By Theorem 3 in \cite{Lu2000} and further cut-off $ B_{K,n}({\bf v},-{\bf v}_*,\og)= B_K({\bf v},-{\bf v}_*,\og)\wedge n$, there exists a conservative solution $f$ such that
	$\sup_{t\in [0,t_1]}\|f(\cdot,t)\|_{L^1_s}<\infty $ for all $t_1>0$ and $s>2$.
	By the same reason in the proof of Theorem 4 in \cite{Lu2004}, we only need to prove the case for $s\ge 4$. We also  need the following version of Povzne-Elmroth inequality (see e.g.\cite{LW} and recall $s\ge 4$.)
	$$\la {\bf v'}\ra^s+\la {\bf v'_*}\ra^s-\la {\bf v}\ra^s-\la {\bf v_*}\ra^s\le 2^{s+1}(\la {\bf v}\ra^{s-1}\la {\bf v}\ra+\la {\bf v}\ra\la {\bf v_*}\ra^{s-1})-2\cos^2\theta\sin^2\theta\la {\bf v}\ra^s. $$
	Now we can compute
	\beas \fr{d\|f(\cdot,t)\|_{L^1_s}}{dt}
	&=&\fr{1}{2}\int_{\mR^3 \times \mR^3 \times S^2}B_Kff_*(\la {\bf v'}\ra^s+\la {\bf v'_*}\ra^s-\la {\bf v}\ra^s-\la {\bf v_*}\ra^s){\rm d}\og {\rm d}{\bf v}{\rm  d}{\bf v}_*\\&+&4\pi \sqrt{2}\int_{\mR^3_+}{\cal K}_{B_K}[\varphi](x,y,z){\rm d}F_t(x){\rm d}F_t(y){\rm d}F_t(z)\\ &\le& 2^{s}\int_{\mR^3\times \mR^3 \times S^2}B_Kff_*(\la {\bf v}\ra^{s-1}\la {\bf v_*}\ra^s+\la {\bf v}\ra\la {\bf v}\ra^{s-1}){\rm d}\og {\rm d}{\bf v} {\rm d}{\bf v}_*\\&-&\int_{\mR^3\times \mR^3 \times S^2}B_K\cos^2\theta \sin^2\theta ff_*\la {\bf v}\ra^s{\rm d}\og {\rm d}v {\rm d}v_* \\&+&\int_{{\mathbb R}_{\ge 0}^3}{\cal K}_{B_K}[\varphi](x,y,z){\rm d}F_t(x){\rm d}F_t(y){\rm d}F_t(z)\\
	&:=& 2^{s}I^{(n)}_1- I^{(n)}_2+4\pi \sqrt{2}I^{(n)}_3,\eeas
	where $\varphi(r)=(1+2r)^{\fr{s}{2}}$.
	Let $A=4\pi \int^{\fr{\pi}{2}}_0 \cos^2\theta \sin^3\theta\min \{\cos^2\theta \sin\theta,\fr{a_0^2}{(4\pi)^2}\cos\theta \}{\rm d}\theta  $ and $C^{(1)}_s$,$C^{(2)}_s$,... denote finite and strictly positve constants that depend only on $a_0,b_1,s$,$R$ and $\beta$. Using the same method \cite{Lu2004} and Appendix of \cite{Cai-Lu} (Equivalence of Solutions), we obtain
	 $$I^{(n)}_1\le C^{(1)}_s\|f(\cdot,t)\|_{L^1_2}\||f(\cdot,t)\|_{L^1_s}.$$
	For $|v-v_*|\ge 2R$, using the condition $ \widehat{\phi}(r) \ge a_0 r^{-\beta}1_{[R, \infty)}(r) $, we calculate
	\beas &&\int_{S^2} B_K({\bf v-v}_*,\og)\cos^2\theta\sin^2\theta {\rm d}\og \\&&\ge |{\bf v-v}_*|\int_{S^2} \cos^2\theta \sin^2\theta\min \{\cos^2\theta \sin\theta,\fr{1}{(4\pi)^2}\cos\theta (\widehat{\phi}(|{\bf v-v}'|)+\widehat{\phi}(|{\bf v-v}_*'|)\big)^2\}{\rm d} \og \\
	&&\ge |{\bf v-v}_*|^{1-2\beta}\int_{S^2} \cos^2\theta \sin^2\theta\min \{\cos^2\theta\sin\theta,\fr{a_0^2}{(4\pi)^2}\cos\theta \}{\rm d}\og
	=|{\bf v-v}_*|^{1-2\beta}A. \eeas
	Using lemma 10 of \cite{Lu2000}, we have
	\beas &&I^{(n)}_2\ge A\int_{|v-v_*|\ge  2R}f({\bf v},t)\la {\bf v}\ra^sf({\bf v}_*,t)|{\bf v-v}_*|A_s {\rm d}{\bf v}{\rm d}{\bf v}_* \\
	&&\ge A \int_{|{\bf v}|>2\sqrt{\fr{E}{N}}+2R}f({\bf v},t)\la {\bf v}\ra^s{\rm d}{\bf v}\int_{|{\bf v}_*|\le 2\sqrt{\fr{E}{N}}}f({\bf v}_*,t)|{\bf v-v}_*|^{1-2\beta}{\rm d}{\rm v}_*  \\
	&&\ge \fr{1}{2}A \int_{|{\bf v}|>2\sqrt{\fr{E}{N}}+2R}f({\bf v},t)\la {\bf v}\ra^s{\rm d}{\bf v}\int_{|{\bf v}_*|\le 2\sqrt{\fr{E}{N}}}f({\bf v}_*)(|{\bf v}|^2+|{\bf v}_*|^2)^{\frac{1}{2}-\beta}{\rm d}{\bf v}_* \\
	&&\ge \fr{3\pi \sqrt{2}NA_s}{2}\int_{|{\bf v}|>2\sqrt{\fr{E}{N}}+2R}f\big(\la {\bf v}\ra^{s+1-2\beta}-\la {\bf v}\ra^{s}\big){\rm d}{\bf v}\\
&&\ge C^{(2)}_s\|f(\cdot,t)\|_{L^1}\big(\|f(\cdot,t)\|_{L^1_{s+1-2\beta}}-(2\sqrt{\fr{E}{N}}+2R+2)\|f(\cdot,t)\|_{L^1_s}\big). \eeas
	Using H{\"o}lder inequality we obtain
	$$\|f(\cdot,t)\|_{L^1_{s+1-2\beta}}\ge (\|f(\cdot,t)\|_{L^1_2})^{-\fr{1-2\beta}{s-2}}(\|f(\cdot,t)\|_{L^1_s})^{1+\fr{1-2\beta}{s-2}}. $$
	This gives
	$$2^{s}I^{(n)}_1-I^{(n)}_2\le C^{(3)}_s\|f(\cdot,t)\|_{L^1_2}\|f(\cdot,t)\|_{L^1_s}  $$
	$$-C^{(4)}_s\|f(\cdot,t)\|_{L^1}(\|f(\cdot,t)\|_{L^1_2})^{-\fr{1-\beta}{s-2}}(\|f(\cdot,t)\|_s)^{1+\fr{1-2\beta}{s-2}}. $$
	For $I_3$, using(\ref{kernel}),(\ref{kernel2}),(\ref{diff1}),(\ref{W1}),(\ref{W2}) and the condition $|\widehat{\phi}(r)|\le b_1$, we deduce that
	$$|{\cal K}_{B_K}[\varphi](x,y,z)|\le 4b_1^2s^2(1+y+z)^{\fr{s}{2}-1}. $$
	So we obtain
	$$4\pi \sqrt{2}I^{(n)}_3\le C^{(5)}_s(\|f(\cdot,t)\|_{L^1})^2\|f(\cdot,t)\|_{L^1_s}.  $$
	
	Now we can see $\|f(\cdot,t)\|_{L^1_s}$ satisfies the following differential inequality
	\begin{align*} \fr{{\rm d}}{{\rm d}t}\|f(\cdot,t)\|_{L^1_s}&\le C^{(6)}_s(1+\|f(\cdot,t)\|_{L^1})\|f(\cdot,t)\|_{L^1_2}\|f(\cdot,t)\|_{L^1_s}\\
	&- C^{(4)}_s\|f(\cdot,t)\|_{L^1}(\|f(\cdot,t)\|_{L^1_2})^{\fr{2\beta -1}{s-2}}(\|f(\cdot,t)\|_{L^1_2})^{1+\fr{1-2\beta}{s-2}}.  \end{align*}
	which implies that(see \cite{Wennberg})
	$$M_s(f(\cdot,t))<\|f(\cdot,t)\|_{L^1_s}\le C^{\fr{s-2}{1-2\beta}}(1+\fr{1}{a})^{\fr{s-2}{1-2\beta}}(1+\fr{1}{t})^{\fr{s-2}{1-2\beta}},  $$
	where
	\beas &&a=\fr{1-2\beta}{s-2}C^{(6)}_s(1+\|f(\cdot,t)\|_{L^1})\|f(\cdot,t)\|_{L^1_2}, \\
	&&C=\fr{C^{(6)}_s(1+\|f(\cdot,t)\|_{L^1})(\|f(\cdot,t)\|_{L^1_2})^{1+\fr{1-2\beta}{s-2}}}{C^{(4)}_s\|f(\cdot,t)\|_{L^1}}. \eeas
This gives the estimate (\ref{mprod}). Having proven the moment production, the proof of the uniqueness is then completely the same as that of Theorem 3 in \cite{Lu2000}.
\end{proof}

\begin{proposition}\lb{prop3.2}  Let the collision kernel $B({\bf {\bf v-v}_*},\omega)$ is given by
	(\ref{kernel}),(\ref{kernel2}),(\ref{Phi})  where the Fourier transform
	$r\mapsto \widehat{\phi}(r)$ is strictly positive in $(0,\infty)$, and satisfies (\ref{3.10}). Given any $N>0, E>0,t_0>0$.
	
	{\rm(I)} Let $B_K({\bf {\bf v-v}_*},\omega)$
	be given by (\ref{3.1}),
	let $\{f_0^K=f^K_0(|{\bf v}|^2/2)\}_{K\in{\mN}} $ be any sequence of nonnegative
	isotropic functions in $L^1_2({\bR})$  satisfying
	\be \int_{{\bR}}(1, |{\bf v}|^2/2)f^K_0(|{\bf v}|^2/2){\rm d}{\bf v}= 4\pi\sqrt{2}(N, E)\qquad
	\forall\, K\in{\mN}.\lb{entropy.3}\ee
	Then for every $K\in{\mN}$, let 	$f^K=f^K(|{\bf v}|^2/2,t)$  be the unique conservative isotropic approximate  solution  of
	Eq.(\ref{Equation1}) on ${\bR}\times [0,\infty)$
	$f^K=f^K(|{\bf v}|^2/2,t)$  corresponding to the approximate kernel
	$B_K$,we have the positive lower bound of entropy as follows:
	\be S(f^K(t))\ge S(f^K(t_0))\ge S_{*}(t_0)\qquad \forall\, t\ge t_0,\,\,\forall\, K\in {\mN}.
	\lb{entropy3.5}\ee
	Where
	\be S_{*}(t_0)=\min\bigg\{\fr{7\pi a^3}{24},\,\fr{4\pi^2 E^2}{5C(1+2/t_0)^2},\,\,
	\min\Big\{
	4m^2,\,\, (4\pi)^2a^2\Big\}
	\fr{7\pi^4 \sqrt{2}a^3E^5 t_0}{96C^3(1+2/t_0)^6}
	\bigg\}\lb{entropy.4}\ee
	and $a=\fr{1}{2}\sqrt{E/N},b=\Big(\fr{C}{2\pi\sqrt{2} E} (1+2/t_0)^2\Big)^{1/2}$, $m=\inf\limits_{a\le \le 2a+b}\widehat{\phi}(r)>0$, $0<C=C_4<\infty$ is the constant in (\ref{mprod}) for $s=4$ so that
	$C$ depends only on $N, E,a_0,b_1$,R and $\beta$.

	{\rm(II)}
Let $F_0\in {\mathcal B}_1^{+}({\mathbb R}_{\ge 0})$ satisfy
$N(F_0)=N, E(F_0)=E$.  Then there exists a conservative  measure-valued isotropic solution $F_t$ of Eq.(\ref{Equation1}) on $[0,\infty)$ with the initial datum $F_0$, such that $F_t$ satisfies moment production (\ref{Mprod}) and
\be S(F_t)\ge S_{*}(t_0)\qquad \forall\, t\ge t_0\lb{entropy3.6}\ee
for all $t_0>0$.
\end{proposition}
\noindent\begin{proof}
Part {\rm (I)}: By Proposition \ref{Mproduct}, $f^K$ is unique and satisfies moment production (\ref{mprod}).
	So we only need to prove the positive lower bound of entropy. As the same in Proposition (\ref{Mproduct}), we omit the superscript $K$ in $f^K$. Since $t\mapsto S(f(t))$ is non-decreasing, it is sufficient to
	prove $S(f(t_0))\ge S_{*}(t_0)$.  To do this we may assume that
	\be S(f(t_0))\le \min\Big\{\fr{7\pi a^3}{24},\,\fr{4\pi^2 E^2}{5C(1+2/t_0)^2}\Big\}.\lb{entropy3.8} \ee
	Let
	\beas&& {\cal V}_t
	=\Big\{({\bf v},{\bf v}_*, \og)\in {\bRRS}\,\,\Big|\,\, a/2\le |{\bf v}|\le a, 2a\le
	|{\bf v}'|\le b, 2a\le |{\bf v}_*'|\le b,\\
	&& f(|{\bf v}|^2/2,t)\le 1/3,
	f(|{\bf v}'|^2/2,t)\ge 9,  f(|{\bf v}_*'|^2/2,t)\ge 9
	\Big\},\quad t\ge t_0/2.
	\eeas
	Then for all $({\bf v},{\bf v}_*, \og)\in {\cal V}_t$ we have
	$a\le |{\bf v}-{\bf v}'|\le 2a+b,a\le |{\bf v}-{\bf v}_*'|\le 2a+b$
	and so
$$ B_K({\bf {\bf v-v}_*},\og)
	\ge \fr{|({\bf v}-{\bf v}_*)\cdot\og|}{(4\pi)^2}\min\Big\{
	4m^2,\, (4\pi)^2 a^2\Big\}.$$
	Using the same method in Proposition 3.4 of \cite{Cai-Lu}, we obtain for $t\ge t_0/2$
	\beas&&
	D_K(f(t))\ge \fr{1}{4}
	\int_{{\cal V}_t}B_K({\bf
		{\bf v-v}_*},\og)\Pi(f)
	\Gm\big({g}'{g}_*',\, g{g}_*\big) {\rm d}\og{\rm
		d}{\bf v_*}{\rm
		d}{\bf v}
	\\
	&&\ge\fr{\min\big\{
		4m^2,\, (4\pi)^2 a^2\big\}}{8(4\pi)^2b^2}
	\int_{{\cal V}_t}|({\bf v}-{\bf v}_*)\cdot\og||{\bf v}'|
	f(|{\bf v}'|^2/2,t)|{\bf v}_*'|f(|{\bf v}_*'|^2/2,t) {\rm d}\og{\rm
		d}{\bf v_*}{\rm
		d}{\bf v}
	\\&& \ge \fr{\min\big\{
		4m^2,\, (4\pi)^2 a^2\big\}}{8(4\pi)^2b^2}\times \fr{7\pi a^3}{12}\Big(\fr{4\pi^2 E^2}{C_4(1+2/t_0)^2}
	\Big)^2	.\eeas	
	Thus we have
	\beas&&
	D_K(f(t))
	\ge \fr{\min\big\{
		4m^2,\, (4\pi)^2 a^2\big\}}{8(4\pi)^2b^2}\fr{7\pi a^3}{12}\Big(\fr{4\pi^2 E^2}{C(1+2/t_0)^2}
	\Big)^2\\
	&&=\min\Big\{
	4m^2),\, (4\pi)^2a^2\Big\}
	\fr{7\pi^4 \sqrt{2}a^3E^5}{48 C^3(1+2/t_0)^6},\\ \\
	&&
	S(f(t_0))=
	S(f(t_0/2))
	+\int_{t_0/2}^{t_0}
	D_K(f(t)){\rm d}t\ge \int_{t_0/2}^{t_0}
	D_K(f(t)){\rm d}t\\
	&&\ge \min\Big\{
	4m^2,\, (4\pi)^2a^2\Big\}
	\fr{7\pi^4 \sqrt{2}a^3E^5}{48 C^3(1+2/t_0)^6}\fr{t_0}{2}
	\ge S_{*}(t_0).\eeas
	This proves (\ref{entropy3.5}).
	
	Part {\rm (II)}: Since by part (I) we know $f^K$ satisfies the moment production (\ref{mprod}) and positive lower bound of entropy (\ref{entropy3.5}), we can use Lemma 3.2 of \cite{Cai-Lu} and Theorem \ref{weak stability} to prove the result.
\end{proof}
\par
In order to prove the weak or semi-strong convergence to equilibrium, we need to assume that the function
$\widehat{\phi}(r)$ has a lower bound function $a(r)$ which is positive, bounded and non-decreasing in $(0,\infty)$. For instance one may take $a(r)=a_0\fr{r^{\eta}}{1+r^{\eta}}$ for some constants $a_0>0, \eta\ge 1$ so that it includes many cases of balanced potentials, e.g. the case where $\widehat{\phi}(r)$ satisfies $\widehat{\phi}(r)= b_0\fr{r^{\eta}}{1+r^{\eta}}.$
In other words, the following theorem tells us that the long-time weak convergence to equilibrium
still holds for many cases of balanced potentials.
\vskip2mm
\noindent{\bf Proof of Theorem \ref{theorem1.12}.} Denote $\sup\limits_{r\ge 0}\wh\phi(r)=b_1$. Let
	$$B_{\min}({\bf v-v}_*,\og)=\fr{1}{(4\pi)^2}\cos^3(\theta)\sin^3(\theta)(|{\bf v-v}_*|\wedge 1)^3a^{2}(\fr{1}{\sqrt{2}}|{\bf v-v}_*|).$$
	Recalling definition of $B_K({\bf v},-{\bf v}_*,\og)$ (see (\ref{3.1})) and using the inequality
	$\max\{|{\bf v-v}'|,|{\bf v-v}_*'|\}\ge \fr{1}{\sqrt{2}}|{\bf v-v}_*| $, we have
	for all $K\ge b_1^2$ that
	$$B_K({\bf v}-{\bf v}_*,\og)
	\ge B_{\min}({\bf v-v}_*,\og)$$
	Thus we can choose the same approxiamation solution $f^K$ with $K\ge b_1^2$ in Theorem 1 of \cite{Lu2005} (with $\underline{b}(\cos(\theta))
	= \fr{b_1^2}{(4\pi)^2}, \Psi(r)=a^{2}(\fr{1}{\sqrt{2}}r)/b_1^2$). Since we have proved moment production and positive lower bound of entropy, we can show $f^K$ ($K\ge b_1^2$) satisfies the same resutls in Theorem 1 of \cite{Lu2005}. Using Theorem \ref{weak stability} and  Lemma 2.1 of \cite{Cai-Lu} we get the result.
$\hfill\Box$
 	\begin{center}\section { Regularity and Stability}\end{center}
 	In this section, we use Theorem \ref{theorem1.3}(the most important case is $p=\fr{1}{2}$.) and some results of  \cite{Lu2014} to get regularity and stability.\par
 First  we define the working space
 	${\cal B}_{p,1}({\mathbb R}_{\ge 0})$ as
 	$$
 	{\cal B}_{p,1}({\mathbb R}_{\ge 0})=
 	\{ F\in {\cal B}_{1}({\mathbb R}_{\ge 0})\,|\,\,
 	M_{p}(|F|)<\infty\},\quad {\cal B}^{+}_{p,1}({\mathbb R}_{\ge 0})={\cal B}_{p,1}({\mathbb R}_{\ge 0})\cap {\cal B}^{+}({\mathbb R}_{\ge 0}).$$
 	It is easily seen that
 	$$p<q<0\quad \Longrightarrow\quad {\cal B}_{p,1}({\mathbb R}_{\ge 0})
 	\subset {\cal B}_{q,1}({\mathbb R}_{\ge 0}),\quad
 	{\cal B}^{+}_{p,1}({\mathbb R}_{\ge 0})\subset
 	{\cal B}^{+}_{q,1}({\mathbb R}_{\ge 0}).$$
 	Let us define
 	$$M_{p,q}(|F|)=M_{p}(|F|)+M_{q}(|F|),\quad
 	-\infty<p,q<\infty.$$
 	And as usual the notations $ F\otimes G, F\otimes G\otimes H$
 	stand for the product measures of $F,G,H$. As the same in \cite{Lu2014}, we introduce the following lemma.
 \vskip2mm
 \begin{lemma} \lb{lemma 2-1} Let the collision kernel $B({\bf {\bf v-v}_*},\omega)$ be given by
 	(\ref{kernel}),(\ref{kernel2}),(\ref{Phi}). Assume $|\widehat{\phi}(r)| \le b_0$ on  $[0,\infty)$, we have
 	\par
 	{\rm (a)}
Let $F,G, H\in {\cal B}_{-1/3,1}({\mathbb R}_{\ge 0})$, $k\in [0,1]$.
Then
$$
\int_{{\mathbb R}_{\ge 0}^3}W{\rm d}(|F|\otimes |G|\otimes |H|)
\le 4b_0^2M_{-1/3}(|F|)M_{-1/3}(|G|)M_{-1/3}(|H|),
$$
$$
\int_{{\mathbb R}_{\ge 0}^3}(1+y^k+z^k)W{\rm d}(|F|\otimes |G|\otimes |H|)
\le4b_0^2 M_{-1/3}(|F|)M_{-1/3,k-1/3}(|G|)M_{-1/3, k-1/3}(|H|).
$$
Furthermore, if  $F,G, H\in {\cal B}_{-1/2,1}({\mathbb R}_{\ge 0})$, then
\be
\int_{{\mathbb R}_{\ge 0}^3}(1+y^k+z^k)W{\rm d}(|F|\otimes |G|\otimes |H|)
\le a(F,G,H)\min\{\|F\|_k, \|G\|_k,\|H\|_k\}\lb{Wi}\ee
where
\be a(F,G,H):=4b_0^2 [ M_{-1/2, 1/2}(|F|)+
M_{-1/2,1/2}(|G|)+M_{-1/2,1/2}(|H|)]^2.
\lb{a-1}\ee

{\rm (b)}  Let  $\vp$ be any
Borel function on ${\mathbb R}_{\ge 0}$ satisfying
$\sup\limits_{x\ge 0}|\vp(x)|(1+x)^{-k}\le 1$ with $k\in [0, 1]$.
Then for all $F,G\in {\cal B}_{k+1/2}({\mathbb R}_{\ge 0})$,
\be \int_{{\mathbb R}_{\ge 0}^2}|{\cal J}^{\pm}[\vp]|{\rm d}(|F|\otimes |G|)\le 4b_0^2 \|F\|_{k+1/2}\|G\|_{k+1/2}, \lb{Jpm2}\ee
\be \int_{{\mathbb R}_{\ge 0}^3}|{\cal K}^{\pm}[\vp]|{\rm d}(|F|\otimes |G|\otimes |H|) \le 8b_0^2M_{-1/3}(|F|)M_{-1/3, k-1/3}(|G|)M_{-1/3,k-1/3}(|H|).\lb{Kpm1}\ee
Furthermore, if  $F,G, H\in {\cal B}_{-1/2,1}({\mathbb R}_{\ge 0})$, then
\be \int_{{\mathbb R}_{\ge 0}^3}|{\cal K}^{\pm}[\vp]|{\rm d}(|F|\otimes |G|\otimes |H|) \le 2a(F,G,H)\min\{\|F\|_k, \|G\|_k,\|H\|_k\}. \lb{Kpm2}\ee
\end{lemma}

\noindent\begin{proof}This is an immediate consequence of Lemma 2.1 in \cite{Lu2014} and the fact that $W(x,y,z)\le 4b^2_0W_H(x,y,z)$.
\end{proof}
\vskip2mm

By Lemma \ref{lemma 2-1},as the same in \cite{Lu2014}, we can define
Borel measures ${\cal Q}_2^{\pm}(F,G)\in {\cal B}_k({\mathbb R}_{\ge 0})$ for $F,G\in {\cal B}_{k+1/2}({\mathbb R}_{\ge 0})\,(k\in[0,1])$ and ${\cal Q}_3^{\pm}(F,G,H)\in
{\cal B}_1({\mathbb R}_{\ge 0})$ for $F,G, H\in {\cal B}_{-1/3, 1}({\mathbb R}_{\ge 0})$
through Riesz representation theorem by
\bes&& \int_{{\mathbb R}_{\ge 0}}\vp(x){\rm d}
{\cal Q}_2^{\pm}(F,G)(x)
=\int_{{\mathbb R}_{\ge 0}^2}{\cal J}^{\pm}[\vp]{\rm d}(F\otimes G),\lb{Q2}\\
&& \int_{{\mathbb R}_{\ge 0}}\vp(x){\rm d}
{\cal Q}_3^{\pm}(F,G, H)(x)
=\int_{{\mathbb R}_{\ge 0}^3}{\cal K}^{\pm}[\vp]{\rm d}(F\otimes G\otimes H)\dnumber\lb{Q3}
\ees
for all
$\vp\in C_b({\mathbb R}_{\ge 0})$.
It is obvious that $(F,G)\mapsto {\cal Q}_2^{\pm}(F,G)$ and  $(F,G, H)\mapsto {\cal Q}_3^{\pm}(F,G, H)$
are bounded bilinear and trilinear operators from  $[{\cal B}_{k+1/2}({\mathbb R}_{\ge 0})]^2
$ to
${\cal B}_{k}({\mathbb R}_{\ge 0})$ and from $
[{\cal B}_{-1/3, 1}({\mathbb R}_{\ge 0})]^3$ to ${\cal B}_{1}({\mathbb R}_{\ge 0})$
respectively ($k\in[0,1]$)  and
\bes&& \|{\cal Q}_2^{\pm}(F,G)\|_k
\le  4b_0^2\|F\|_{k+1/2}\|G\|_{k+1/2}, \lb{Q2single}\\
&&\|{\cal Q}_3^{\pm}(F,G, H)\|_0\le
8b_0^2M_{-1/3}(|F|)M_{-1/3}(|G|)M_{-1/3}(|H|), \dnumber\lb{Q3single0}
\\
&&\|{\cal Q}_3^{\pm}(F,G, H)\|_k\le
8b_0^2M_{-1/3}(|F|)M_{-1/3, k-1/3}(|G|)M_{-1/3,k-1/3}(|H|), \dnumber\lb{Q3single00}
\\
&&
\|{\cal Q}_3^{\pm}(F,G, H)\|_k\le  2a(F,G,H)
\min\{\|F\|_k, \|G\|_k,\|H\|_k\}.\dnumber\lb{Q3single}\ees
Here in the third inequality (\ref{Q3single}) we assume further that
$F,G, H\in {\cal B}_{-1/2,1}({\mathbb R}_{\ge 0})$ so that
$a(F,G,H)<\infty$.

In connecting with the equation Eq.(\ref{weak}) we define
$$ {\cal Q}_2^{\pm}(F)
={\cal Q}_2^{\pm}(F,F),\qquad{\cal Q}_2(F)={\cal Q}_2^{+}(F)-{\cal Q}_2^{-}(F),
$$
$${\cal Q}_3^{\pm}(F)={\cal Q}_3^{\pm}(F, F, F),\quad {\cal Q}_3(F)=
{\cal Q}_3^{+}(F)-{\cal Q}_3^{-}(F),$$
$${\cal Q}(F)={\cal Q}_2(F)+{\cal Q}_3(F).$$
We then deduce from
$$F\otimes F-G\otimes G=\fr{1}{2}(F-G)\otimes(F+G)+\fr{1}{2}
(F+G)\otimes(F-G),$$
$${\cal Q}_2^{\pm}(F)-{\cal Q}_2^{\pm}(G)
=\fr{1}{2}{\cal Q}_2^{\pm}(F-G,F+G)
+\fr{1}{2}{\cal Q}_2^{\pm}(F+G,F-G),
$$
and (\ref{Q2single}) that for all $F,G\in {\cal B}_{k+1/2}({\mathbb R}_{\ge 0})$ (with $k\in [0,1]$)
\be \|{\cal Q}_2^{\pm}(F)-{\cal Q}_2^{\pm}(G)\|_k\le 4b_0^2 \|F+G\|_{k+1/2}\|F-G\|_{k+1/2}.
\lb{Q2-1}\ee
Similarly we deduce from
$$F\otimes F\otimes F-G\otimes G\otimes G
=(F-G)\otimes F\otimes F +G\otimes(F-G)\otimes F
+G\otimes G\otimes (F-G),$$
$$\|{\cal Q}_3^{\pm}(F)-{\cal Q}_3^{\pm}(G)\|_k
\le \|{\cal Q}_3^{\pm}(F-G, F,F)\|_k+\|{\cal Q}_3^{\pm}(G, F-G,F)\|_k+\|
{\cal Q}_3^{\pm}(G, G,F-G)\|_k$$
and (\ref{Q3single00}), (\ref{Q3single}) that
\bes && \|{\cal Q}_3^{\pm}(F)-{\cal Q}_3^{\pm}(G)\|_0\le
8b_0^2[M_{-1/3}(|F|)+M_{-1/3}(|G|)]^2M_{-1/3}(|F-G|)
,\lb{Q3-0}\\
&&
\|{\cal Q}_3^{\pm}(F)-{\cal Q}_3^{\pm}(G)\|_k\le
 b(F,G)\|F-G\|_k,\quad k\in [0, 1]\dnumber \lb{Q3-1}\ees
where for the inequality (\ref{Q3-1}) we assume that  $F,G\in {\cal B}_{-1/2,1}^{+}({\mathbb R}_{\ge 0})$ so that
\be b(F,G):=144b_0^2[M_{-1/2, 1/2}(|F|)
+M_{-1/2,1/2}(|G|)]^2<\infty.\lb{b-1}\ee
In order to prove Theorem \ref{theorem1.10} and Theorem \ref{theorem1.11}, we shall introduce the concept of strong solutions.
\begin{definition}\label{definition 2-1}
Let $F_t$ be a distributional solution of Eq.(\ref{Equation1})
on $[0,\infty)$. Let $0<T_{\infty}\le\infty$.
We say that $F_t$ is a strong solution of Eq.(\ref{Equation1}) on $[0, T_{\infty})$
if it satisfies the following {\rm (i)}-{\rm(iii)}:

{\rm(i)}\, $t\mapsto F_t$ belongs to $
C^1([0,T_{\infty});{\mathcal B}_{0}({\mathbb R}_{\ge 0}))$,

{\rm (ii)}\, $t\mapsto {\cal Q}_2^{\pm}(F_t), t\mapsto {\cal Q}_3^{\pm}(F_t)$ belong to $
C([0,T_{\infty});{\mathcal B}_{0}({\mathbb R}_{\ge 0}))$,  and

{\rm(iii)}
\be \fr{{\rm d}}{{\rm d}t}F_t={\mathcal Q}(F_t)\quad in\quad ({\cal B}_{0}({\mathbb R}_{\ge 0}), \|\cdot\|_0 )\qquad \forall\,t\in[0, T_{\infty}).\lb{strong-int}\ee

Besides, if $F_t$ also conserves the energy on $[0,T_{\infty})$, then $F_t$ is also called a conservative strong solution of Eq.(\ref{Equation1}) on $[0, T_{\infty})$.

Strong solutions can be also defined on a finite closed time-interval by replacing
$[0,T_{\infty})$ with $[0, T]$ for $0<T<\infty$.
\end{definition}

\begin{remark}\lb{remark 2-1}

{\rm   Under the condition (ii), the conditions {\rm (i),(iii)} are equivalent to the
integral equation:
\be F_t=F_0+\int_{0}^t{\cal Q}(F_\tau){\rm d}\tau\qquad \forall\,t\in[0, T_{\infty})\ee
where the integral is taken as the Riemann integral defined with the norm $\|\cdot\|_0$. This then implies that, under the condition {\rm (ii)}, the integral equation (\ref{strong-int}) is equivalent to its  dual form:
\be
\int_{{\mathbb R}_{\ge 0}}\psi{\rm d}F_t=\int_{{\mathbb R}_{\ge 0}}\psi{\rm d}F_0
+\int_{0}^{t}{\rm d}\tau \int_{{\mathbb R}_{\ge 0}}\psi{\rm d}{\cal Q}(F_\tau)\qquad
\forall\, \psi\in L^{\infty}({\mathbb R}_{\ge 0})
\lb{bd-func}\ee
for all $t\in[0, T_{\infty})$. }
\end{remark}

\begin{proposition}\label{proposition 4-4}
	Suppose $B({\bf {\bf v-v}_*},\omega)$ satisfy  Assumption \ref{assp} with $\eta\ge \fr{3}{2}$. Let $F_t$ be a distributional solution of Eq.(\ref{Equation1})
	on $[0,\infty)$ with the initial datum $F_0$ satisfying
	$M_{-1/2}(F_0)<\infty$.
	Then $F_t$ is a strong solution of Eq.(\ref{Equation1}) on $[0, \infty)$.
\end{proposition}
\noindent \begin{proof}Take any $T\in (0, \infty)$. Using Theorem \ref{theorem1.3} for $p=\fr{1}{2}$ we know $ \sup\limits_{t\in[0,T]}M_{-1/2}(F_t)<\infty$.Combining with the estimates (\ref{Q2single}), (\ref{Q3single}) for $k=1/2$  we have
	$\|{\cal Q}_2^{\pm}(F_t)\|_{1/2},\, \|{\cal Q}_3^{\pm}(F_t)\|_{1/2}\le C_T$ for all $ t\in[0, T]$,
	where $C_T<\infty$ depends only on
	$ \sup\limits_{t\in[0,T]}M_{-1/2}(F_t)<\infty$ and $\sup\limits_{t\in[0,T]}\|F_t\|_1$
	. From this and the integral equation (\ref{Equation3})
	which also reads
	\be \int_{{\mathbb R}_{\ge 0}}\vp{\rm
		d}(F_t-F_s)=\int_{s}^{t}{\rm d}\tau\int_{{\mathbb R}_{\ge 0}}\vp {\rm d}{\cal Q}(F_{\tau})\qquad
	\forall\,\vp\in C_b^{1,1}({\mathbb R}_{\ge 0})\lb{FFQ}\ee
	we obtain
	$\|F_t-F_s\|_0\le C_T|t-s|$ for all $t,s\in[0, T].$ Since $\|F_t-F_s\|_{1/2}\le \|F_t-F_s\|_1^{1/2}\|F_t-F_s\|_0^{1/2}$ by Cauchy-Schwarz inequality,
	it follows that  $t\mapsto F_t$ also belongs to $
	C([0,\infty);{\mathcal B}_{1/2}({\mathbb R}_{\ge 0}))$
	and thus
	we conclude from (\ref{Q2-1})-(\ref{Q3-1}) with $k=0$ that
	$t\mapsto  {\cal Q}_2^{\pm}(F_t), t\mapsto  {\cal Q}_3^{\pm}(F_t)$ hence $t\mapsto  {\cal Q}(F_t)$ all belong to
	$C([0, \infty);{\mathcal B}_{0}({\mathbb R}_{\ge 0}))$.
	Next for any $T\in (0, \infty)$, using 	$ \sup\limits_{t\in[0,T]}M_{-1/2}(F_t)<\infty$ and  smooth approximation it is easily deduced that (\ref{FFQ}) with $s=0$ and $t\in [0, T]$
	holds for all bounded Borel functions $\vp$ on ${\mathbb R}_{\ge 0}$, in particular it holds for all characteristic functions $\vp(x)={\bf 1}_{E}(x)$ of Borel sets $E\subset {\mathbb R}_{\ge 0}$. Therefore $F_t$ satisfies
	the integral equation (\ref{strong-int}) and so,
	according to the equivalent definition of strong solutions discussed in Remark \ref{remark 2-1}, $F_t$ is a strong solutions of Eq.(\ref{Equation1}) on $[0, \infty)$.
\end{proof}
\par
The proofs of Theorem \ref{theorem1.10} and Theorem \ref{theorem1.11} is essentially the same as those of Proposition 4.1 and Theorem 3.1 in \cite{Lu2014}. The only difference is that we can use Theorem \ref{theorem1.3} and Proposition \ref{proposition 4-4} to ensure the propagation of $M_{-1/2}(F_t)$ 
so as to obtain a global in time strong solution. For the sake of completeness, 
we provide complete proofs below.
\vskip2mm
\noindent{\bf Proof of Theorem \ref{theorem1.10}.} Using Theorem \ref{theorem1.3} we know $ \sup\limits_{t\in[0,T]}M_{-1/2}(F_t)<\infty$ for all $T\in [0,\infty)$.
	Recalling Proposition \ref{proposition 4-4} that $F_t$ is a strong
	distributional solution on $[0, \infty)$ and relation (\ref{Mp}) we
	have $F_t(\{0\})=0$ for all $t\in[0,\infty)$, which
	means that the origin $x=0$ has no contribution with respect to the measure $F_t$ and thus the integration domain ${\mathbb R}_{\ge 0}$ can be replaced by ${\mathbb R}_{+}={\mathbb R}_{>0}$.
	Let $$V_t(\dt)=\sup_{{\rm mes}(U)<\dt}F_t(U),\quad t\in [0, \infty) $$
	where $E\subset {\mathbb R}_{\ge 0}$ is any Borel set,
	$U$ is chosen from all open sets in ${\mathbb R}_{+}$, and ${\rm mes}(\cdot)$ denotes the Lebesgue
	measure
	on $t\in [0, \infty)$. Take any open set $U\subset {\mathbb R}_{+}$
	satisfying ${\rm mes}(U)<\dt$.
	Applying the integral equation (\ref{bd-func}) to a monotone sequence $0\le \vp_n\in C_b({\mathbb R}_{\ge 0})$ satisfying
	$$\vp_n(x) \nearrow \psi_U(x):={\bf 1}_{U}(x)\quad(n\to\infty)\quad \forall\,x\in {\mathbb R}_{+}$$
	for instance $
	\vp_n(x)=(1-\exp(-n{\rm dist}(x,U^c)))$, and then omitting negative parts we deduce from monotone convergence that
$$F_t(U)
\le F_0(U)
+\int_{0}^{t}{\rm d}\tau\int_{{\mathbb R}_{+}^2}{\cal J}^{+}
[\psi_U]{\rm d}^2F_\tau+
\int_{0}^{t}{\rm d}\tau\int_{{\mathbb R}_{+}^3}{\cal K}^{+}[\psi_U]{\rm d}^3F_\tau,
\quad t\in [0,\infty)$$
where
\beas&& {\cal J}^{+}[\vp](y,z)=\fr{1}{2}\int_{0}^{y+z}{\cal K}^{+}[\vp](x, y,z)
\sqrt{x}{\rm d}x, \\
&&
{\cal K}^{+}[\vp](x, y,z)=W(x,y,z)[\vp(x)+\vp(x_*)].\eeas
	Next we compute for all $x,y,z>0$
	\beas&& {\cal J}^{+}[\psi_U](y,z)
	\le\fr{1}{2}\int_{0}^{y+z}W(x,y,z)(1_U(x)+1_U(y+z-x)\sqrt{x}{\rm d}x \\
	&&
	\le\int_{0}^{y+z} 2b_0^2\fr{\min\{1,\max\{8x,8y,8z\}^\eta\}}{\sqrt{y}\sqrt{z}}\min\{\sqrt{x},\sqrt{y},\sqrt{z},\sqrt{x_*}\}(1_U(x)+1_U(y+z-x){\rm d}x  ,,\\
	\\
	&&
	\le\int_{0}^{y+z} 2b_0^2\fr{{8^{\eta}(y+z)^{\fr{1}{2}}}}{\sqrt{y}\sqrt{z}}\min\{\sqrt{x},\sqrt{y},\sqrt{z},\sqrt{x_*}\}(1_U(x)+1_U(y+z-x){\rm d}x \le 8^{1+\eta}b_0^2 \delta,
	\eeas

\beas&& \int_{0\le x,y,z}W(x,y,z)1_U(x){\rm d}^3F_{\tau}
	\le 2\bigg(\int_{0\le x\le y \le z}W(x,y,z)1_U(x){\rm d}^3F_{\tau}\\
	&&+\int_{0\le y \le x \le z}W(x,y,z)1_U(x){\rm d}^3F_{\tau}+\int_{0\le y \le z \le x}W(x,y,z)1_U(x){\rm d}^3F_{\tau}\bigg)
	\\
	&&\le 8b_0^2\bigg(\int_{0\le x\le y \le z}\fr{\min\{1,(8z)^\eta\}}{\sqrt{y}\sqrt{z}}1_U(x){\rm d}^3F_{\tau}
	\\
	&&
	+\int_{0\le y \le x \le z}\fr{\min\{1,(8z)^\eta\}}{\sqrt{y}\sqrt{z}}1_U(x){\rm d}^3F_{\tau}+\int_{0\le y \le z \le x}\fr{\min\{1,(8x)^\eta\}}{\sqrt{x}\sqrt{z}}1_U(x){\rm d}^3F_{\tau}\bigg)
	\\
	&&
	\le 8^{1+\eta}b_0^2\bigg(\int_{0\le x\le y \le z}\fr{1}{\sqrt{y}}1_U(x){\rm d}^3F_{\tau}+\int_{0\le y \le x \le z}\fr{1}{\sqrt{y}}1_U(x){\rm d}^3F_{\tau}+\int_{0\le y \le z \le x}\fr{1}{\sqrt{z}}1_U(x){\rm d}^3F_{\tau}\bigg)\\
	&&
	\le 3\cdot 8^{1+\eta}b_0^2NM_{-1/2}(F_t)V_{\tau}(\delta),
	\eeas
and
\beas&&
	\int_{0\le x,y,z}W(x,y,z)1_U(y+z-x){\rm d}^3F_{\tau}
	\le 2(\int_{0\le x\le y \le z}W(x,y,z)1_U(y+z-x){\rm d}^3F_{\tau}\\
	&&+\int_{0\le y \le x \le z}W(x,y,z)1_U(y+z-x){\rm d}^3F_{\tau}+\int_{0\le y \le z \le x}W(x,y,z)1_U(y+z-x){\rm d}^3F_{\tau})
	\\
	&&
	\le 8b_0^2(\int_{0\le x\le y \le z}\fr{\min\{1,(8z)^\eta\}}{\sqrt{y}\sqrt{z}}1_U(y+z-x){\rm d}^3F_{\tau}
	\\
	&&
	+\int_{0\le y \le x \le z}\fr{\min\{1,(8z)^\eta\}}{\sqrt{y}\sqrt{z}}1_U(y+z-x){\rm d}^3F_{\tau}+\int_{0\le y \le z \le x}\fr{\min\{1,(8x)^\eta\}}{\sqrt{x}\sqrt{z}}1_U(y+z-x){\rm d}^3F_{\tau})
	\\
	&&
	\le 8^{1+\eta}b_0^2(\int_{0\le x\le y \le z}\fr{1}{\sqrt{y}}1_U(y+z-x){\rm d}^3F_{\tau}\\
	&&+	\int_{0\le y \le x \le z}\fr{1}{\sqrt{y}}1_U(y+z-x){\rm d}^3F_{\tau}+\int_{0\le y \le z \le x}\fr{1}{\sqrt{z}}1_U(y+z-x){\rm d}^3F_{\tau})\\
	&&\le 3\cdot 8^{1+\eta}b_0^2NM_{-1/2}(F_t)V_{\tau}(\delta).
	\eeas
	It follows that
	$$
	F_t(U)
	\le V_0(\dt)+8^{1+\eta}b_0^2 \delta N^2 t+8^{2+\eta}b_0^2N
	\int_{0}^{t}M_{-1/2}(F_{\tau})V_{\tau}(\delta){\rm d}\tau.$$
	Taking $\sup\limits_{{\rm mes}(U)<\dt}$ leads to
	$$V_t(\dt)\le   V_0(\dt)+8^{1+\eta}b_0^2 \delta N^2 t+8^{2+\eta}b_0^2N
	\int_{0}^{t}M_{-1/2}(F_{\tau})V_{\tau}(\delta){\rm d}\tau
	,\quad t\in [0,\infty)$$
	and so, by Gronwall inequality,
	$$V_t(\dt)\le \Big( V_0(\dt)+8^{1+\eta}b_0^2 \delta N^2 t
	\Big)\exp\Big(8^{2+\eta}b_0^2N\int_{0}^{t}M_{-1/2}(F_{\tau}){\rm d}\tau\Big),\quad t\in [0,\infty).$$
	Since $F_0$ is regular implies $\lim\limits_{\dt\to 0^+}V_0(\dt)=0$ and since 
$t\mapsto M_{-1/2}(F_{t})$ is locally bounded on $[0,\infty)$ (see Theorem \ref{theorem1.3}),
it follows that
	$\lim\limits_{\dt\to 0^+}V_t(\dt)=0$ for all $t\in[0,\infty).$
This proves that
	$F_t$ is absolutely continuous with respect to the Lebesgue measure for every $t\in[0,\infty)$, and
	thus there is a unique $0\le f(\cdot,t)\in L^1({\mathbb R}_{\ge 0})$ such that
	${\rm d}F_t(x)=f(x,t)\sqrt{x}{\rm d}x.$
	That is, we have proved that $F_t$ is regular for all $t\in[0, \infty)$
	and its density $f(\cdot,t)$ belongs to $L^1({\mathbb R}_{\ge 0})$ for all  $t\in[0, \infty)$.
	
Since $\|f(t)\|_{L^1}=
	M_{-1/2}(F_t)$, it follows from $ \sup\limits_{t\in[0,T]}M_{-1/2}(F_t)<\infty$ 
and $W(x,y,z)\sqrt{y}\sqrt{z}\le 4b^2_0W_H(x,y,z)\sqrt{y}\sqrt{z}=4b^2_0\fr{\min\{\sqrt{x},\sqrt{x_*},
		\sqrt{y},\sqrt{z}\,\}}{\sqrt{x}} $ that for all $0<T<\infty$
	\bes &&
	\sup_{0\le t\le T}\int_{{\mathbb R}_{+}^3}W(x,y,z)[ f'f_*'(1+f+f_*))+
	ff_*(1+f'+f_*')]\sqrt{y}\sqrt{z}{\rm d}x{\rm d}y{\rm
		d}z\nonumber \\
	&&
	\le 16b^2_0\sup_{0\le t\le T}\Big(
	M_{1/2}(f(t))\|f(t)\|_{L^1}+\|f(t)\|_{L^1}^3\Big)<\infty\lb{Qf}\ees
	where $M_{1/2}(f(t))=\int_0^{\infty}xf(x,t){\rm d}x$. From this and that $F_t$ is a strong solution of Eq.(\ref{Equation1}) on $[0, \infty)$
	we conclude that the equation
	\beas&& \int_{{\mathbb R}_{\ge 0}}\psi(x)\Big(
	f(x,t)-f_0(x)-\int_{0}^{t}Q(f)(x,\tau){\rm d}\tau
	\Big)\sqrt{x}\,
	{\rm d}x\\
	&& =\int_{{\mathbb R}_{\ge 0}}\psi(x){\rm d}\Big(F_t
	-F_0
	-\int_{0}^{t}{\cal Q}(F_\tau){\rm d}\tau\Big)(x)=0\eeas
	holds for all $t\in[0, \infty)$ and all bounded Borel functions $\psi$ on ${\mathbb R}_{\ge 0}$. Thus
	for any $t\in [0, \infty)$, there is a null set $Z_t\subset
	{\mathbb R}_{\ge 0}$ such that
	$$f(x,t)=f_0(x)+\int_{0}^{t}Q(f)(x,\tau){\rm d}\tau
	\qquad \forall\, x\in{\mathbb R}_{\ge 0}\setminus Z_t.$$
	In order to get a common null set $Z$ independent of $t$, we consider
	$\wt{f}(\cdot,t):=|f_0+\int_{0}^{t}Q(f)(\cdot,\tau){\rm d}\tau
	|$. The advantage of $\wt{f}(\cdot,t)$ is that there is a null set $Z$ which is independent of $t$ such
	that $t\mapsto \wt{f}(x,t)$ is continuous in $t\in[0, \infty)$ for all $x\in {\mathbb R}_{+}\setminus Z$. Also, since $\wt{f}(x,t)=f(x,t)$ for all $t\in[0, \infty)$ and all  $x\in{\mathbb R}_{+}\setminus Z_t$, it follows
	from Fubini theorem  that
	$\wt{f}(\cdot,t)$ is a mild solution to Eq.(\ref{Equation1}) on
	$[0, \infty)$. Again since
	$\wt{f}(\cdot,0)=f_0$ and  $\wt{f}(x,t)=f(x,t)$ for all $t\in[0, \infty)$ and all $x\in {\mathbb R}_{+}\setminus Z_t$, it follows that
	$\wt{f}(\cdot,t)$ is also the same density of $F_t$ for $t\in[0, \infty)$.
	Thus by rewriting $\wt{f}(\cdot, t)$ as $f(\cdot,t)$  we conclude that
	the density $f(\cdot,t)$ of $F_t$ is a mild solution of Eq.(\ref{Equation1}) on
	$[0, \infty)$.
	
	Finally for any $T\in (0, \infty)$, let $C_T$ be the left hand side of (\ref{Qf}). Then
	we deduce from (\ref{Qf}) and the definition of mild solutions that
	$\|f(\tau)-f(t)\|_{L^1}\le 2C_T |s-t|$  for all $s, t\in [0, T]$. Therefore $f\in C([0, \infty); L^1({\mathbb R}_{\ge 0}))$.
$\hfill\Box$
\vskip3mm

\noindent{\bf Proof of Theorem \ref{theorem1.11}.} First according to Proposition \ref{proposition 4-4},
		$F_t, G_t$ are strong solutions on $[0, \infty)$.
		The proof is divided into three steps. First we assume that $F_t$ has the moment production (\ref{Mprod}) for all $t\in (0, \infty)$. The existence of such $F_t$ is
		assured by Proposition \ref{prop3.2}. Let us denote
		$$H_t=F_t-G_t.$$
		By
		conservation of mass we have
		$\|F_t\pm G_t\|_1\le \|F_0\|_1+\|G_0\|_1$ for all $t\ge 0$.
		So if $\|H_0\|_1\ge 1$, then
		$\|H_t\|_1\le 2\|F_0\|_1+\|H_0\|_1\le (2\|F_0\|_1+1)\|H_0\|_1$
		for all $t\ge 0$.
		Therefore to prove (\ref{H0}) we can assume $\|H_0\|_1<1$.
		
		{\bf Step 1.} Given any $s\in (0, t)$,we prove that
		\bes&&\|H_t\|_0\le  \|H_0\|_0
		+C_1(t)\int_{0}^{t}\|H_\tau\|_1{\rm d}\tau ,\lb{H1}\\
		&&
		\|H_t\|_1\le \|H_s\|_1+
		C_0\int_{s}^{t}(1+1/\tau)\|H_\tau\|_{0}{\rm d}\tau+C_1(t)
		\int_{s}^{t}\|H_\tau\|_1{\rm d}\tau.  \dnumber \lb{H2}\ees
		Here and below the constant $0<C_0<\infty$ depends only
		on $N(F_0)$ ,$E(F_0)$,$\beta,a_0,b_0$ and $R$,
and  
\beas&& C_1(t)=288b_0^2(2at+M_{-1/2}(F_0)+M_{-1/2}(G_0)+\|F_0\|_1+\|G_0\|_1+1)^2e^{2bt},\\
&& a=8^2b_0^2 \max\{N(F_0),N(G_0)\}^{2}+8^{2+\eta}b_0^2\max\{N(F_0),N(G_0)\}^3,\\
&&b=8^{3+\eta}b_0^2\max\{N(F_0),N(G_0)\}^2(1+q_1).\eeas
The inequality (\ref{H1}) follows from $H_t=H_0+\int_{0}^{t}[{\cal Q}(F_\tau)-{\cal Q}(G_\tau)]{\rm d}\tau$, Theorem \ref{theorem1.3} and the estimates (\ref{Q2-1}), (\ref{Q3-1}) for $k=0$.  To prove (\ref{H2}) we first use the identity $|H_t|=-H_t+2(H_t)_{+}$ (recall that $H_t=F_t-G_t$) and the conservation of mass and energy to write
		\be \|H_t\|_1=\|G_s\|_1-\|F_s\|_1+2\|(H_t)_{+}\|_1,\quad t\ge s.\lb{Hlimit}\ee
		Let $
		x\mapsto \kappa_t(x)\in\{0,1\}$ be the Borel function on ${\mathbb R}_{\ge 0}$
		such that
		$\kappa_t(x){\rm d}H_t(x)={\rm d}(H_t)_{+}(x)$. Since $t\mapsto {\cal Q}(F_t)-{\cal Q}(G_t)$
		belongs to
		$C([0, \infty); {\cal B}_0({\mathbb R}_{\ge 0}))$,
		applying Lemma 5.1 of \cite{LM} to the measure equation
		$
		H_t=H_s+\int_{s}^t ( {\cal Q}(F_\tau)-{\cal Q}(G_\tau)){\rm d}\tau,\, t\in [s, \infty),
		$ we have
		$$\int_{{\mathbb R}_{\ge 0}}\psi(x){\rm d}(H_t)_{+}(x)=
		\int_{{\mathbb R}_{\ge 0}}\psi(x){\rm d}H_s(x)
		+\int_{s}^t{\rm d}\tau \int_{{\mathbb R}_{\ge 0}}\psi(x)
		\kappa_\tau(x){\rm d}( {\cal Q}(F_\tau)-{\cal Q}(G_\tau))(x) $$
		for all $t\in [s, \infty)$ and all bounded Borel functions $\psi$ on
		${\mathbb R}_{\ge 0}$.
		In particular  we have
		$$
		\int_{\mathbb{R}_{\ge 0}}(1+x \wedge n) {\rm d}(H_t)_{+}(x) \le \|(H_s)_{+}\|_1+\int_{s}^{t}{\rm
			d}\tau\int_{\mathbb{R}_{\ge 0}}(1+x \wedge n)\kappa_\tau(x)
		{\rm d}( {\cal Q}(F_\tau)-{\cal Q}(G_\tau))(x).$$
		Next applying (\ref{Mprod}) with $p=3/2$ we see that the function
		$t\mapsto \|F_t\|_{3/2}\le C_{0}(1+1/t)$
		is integrable on $[s, T]$ . Using the estimate that is analogous to Lemma 3.5 of \cite{Lu2014}  and the reverse Fatou's Lemma we deduce
		\beas&& \limsup_{n\to\infty}\int_{s}^{t}{\rm d}\tau\int_{{\mathbb R}_{\ge 0}}(1+x\wedge n)\kappa_\tau(x){\rm d}({\cal Q}(F_\tau)-
		{\cal Q}(G_\tau))(x)\\
		&&  \le C_0
		\int_{s}^{t}(1+1/\tau)\|H_\tau\|_0{\rm d}\tau+C_1(t)\int_{s}^{t}\|H_\tau\|_1{\rm d}\tau.\eeas
		Letting $n\to\infty$  we conclude
		$$\|(H_t)_{+}\|_1\le \|(H_s)_{+}\|_1+ C_0\int_{s}^{t}(1+1/\tau)\|H_\tau\|_0{\rm d}\tau
		+C_1(t)\int_{s}^{t}\|H_\tau\|_1{\rm d}\tau.$$
		This together with (\ref{Hlimit}) and
		$\|G_s\|_1-\|F_s\|_1+2\|(H_s)_{+}\|_1
		=\|H_s\|_1$
		gives (\ref{H2}).

		{\bf Step 2.} We prove that  for any $R_1\ge 1$
		\be \|H_t\|_1\le 5R_1\|H_0\|_1+C_1(t)Rt
		+2\int_{x>R_1}x {\rm d}F_0(x) \lb{H3}\ee
		In fact using $|H_t|=G_t-F_t+2(H_t)_{+}$ and conservation of mass and energy we have
		\be
		\|H_t\|_1\le \|H_0\|_1+ 4R_1\|H_t\|_0
		+2\int_{x> R_1}x{\rm d}F_t(x)
		\lb{H4}\ee
		and applying (\ref{bd-func}) to the bounded function $\psi(x)= {\bf 1}_{\{x\le R_1\}}x$ we deduce
		\beas&& \int_{x> R_1}x{\rm d}F_t(x)
		=E(F_0)-\int_{{\mathbb R}_{\ge 0}}{\bf 1}_{\{x\le R_1\}}x{\rm d}F_t(x)
		\\
&& =\int_{x> R}x{\rm d}F_0(x)
		-\int_{0}^{t}{\rm d}\tau
		\int_{{\mathbb R}_{\ge 0}}{\bf 1}_{\{x\le R_1\}}x {\rm d}{\cal Q}(F_\tau)(x)\\
		&&\le \int_{x> R_1}x{\rm d}F_0(x)+
		R_1\int_{0}^{t}\|{\cal Q}(F_\tau)\|_0{\rm d}\tau
		\le \int_{x> R_1}x{\rm d}F_0(x)+C_1(t) R_1 t.
		\eeas
		This together with (\ref{H4}) and $\|H_t\|_0\le \|H_0\|_0
		+C_1(t) t$ (by (\ref{H1})) yields (\ref{H3}).

		{\bf Step 3.}
		If $t\le \|H_0\|_1$, we take
		$R_1=\fr{1}{\sqrt{\|H_0\|_1}}$ and use (\ref{H3}) to get
		$$\|H_t\|_1
		\le C_2(t)\Big(\sqrt{\|H_0\|_1}
		+\int_{x> \fr{1}{\sqrt{\|H_0\|_1}}}x{\rm d}F_0(x)
		\Big)
		\le C_2(t))\Psi_{F_0}(\|H_0\|_1)$$
		where $C_2(t)=C_1(t)+5$. Suppose now
		$\|H_0\|_1<t$ and let $\vep>0$ satisfy $\|H_0\|_1\le \vep< 1.$
		Taking $R_1=\fr{1}{\sqrt{\vep}}$ and using (\ref{H3}) we have
		\be \|H_\tau\|_1\le  C_2(\tau)\sqrt{\vep}\,+2\int_{x>\fr{1}{\sqrt{\vep}}}x {\rm d}F_0(x)
		\le C_2(\tau)\Psi_{F_0}(\vep),
		\quad \forall\, \tau \in [0, \vep].\lb{H5}\ee
		In particular this inequality holds for $\tau=\vep$. Thus
		using (\ref{H2}) for $s=\vep$ gives
		\be
		\|H_t\|_1\le C_2(\vep)\Psi_{F_0}(\vep)+C_0\int_{\vep}^{t}(1+1/\tau)\|H_\tau\|_{0}{\rm d}\tau+
		C_2(t)\int_{\vep}^{t}\|H_\tau\|_1{\rm d}\tau, \quad t\in[\vep, 1]\lb{H*}\ee
		Next, using (\ref{H1}) we know for $\|H_0\|_0\le \vep\le t\le 1$,
		\bes
		\int_{\vep}^{t}(1+1/\tau)\|H_\tau\|_{0}{\rm d}\tau
		\!\!&\le &\!\! 2\vep\log(1/\vep)
		+2C_1(1)\int_{\vep}^{t}\fr{1}{\tau}\,\int_{0}^{\tau}\|H_u\|_1{\rm d}u {\rm d}\tau
		\nonumber\\
		\!\!&\le &\!\!  2\sqrt{\vep}+2C_1(1)\int_{0}^{t} \|H_u\|_1|\log u|
		{\rm d}u,\quad t\in [\vep, 1].
		\nonumber\ees
		This together with (\ref{H*}) and (\ref{H5}) gives
		\be
		\|H_t\|_1\le 2C_2(1)\Psi_{F_0}(\vep)+2(C_2(1))^2
		\int_{0}^{t}(1+|\log \tau|)\|H_\tau\|_1{\rm d}\tau,\quad t\in[0,  1].\lb{H6}\ee
		By Gronwall inequality we then obtain
		\be \|H_t\|_1\le 2C_2(1)\Psi_{F_0}(\vep)\exp\Big(2(C_2(1))^2\int_{0}^{t}(1+|\log \tau|){\rm d}\tau\Big)
		= C_3\Psi_{F_0}(\vep),\quad t\in [0,  1]\lb{H7}\ee
		where $C_3=2C_2(1)\exp\Big(2(C_2(1))^2\int_{0}^{t}(1+|\log \tau|){\rm d}\tau\Big)$. Now if $t\le 1$, then (\ref{H0}) follows from (\ref{H7}). Suppose $t>1$, using (\ref{H7}) we know $\|H_1\|_1
		\le C_3\Psi_{F_0}(\vep)$.
		On the other hand from (\ref{H2}) with $s=1$ we have
		$\|H_t\|_1\le \|H_1\|_1+
		(C_1(t)+2C_0)\int_{1}^{t}\|H_\tau\|_1{\rm d}\tau $ for all $t\in[1,\infty]$ and
		so
		$\|H_t\|_1\le \|H_1\|_1 e^{c(t-1)}\le C_3\Psi_{F_0}(\vep)e^{	(C_1(t)+2C_0)t}$ for all $t\in[1,\infty)$ by Gronwall Lemma.
		This together with the estimate for $t\in [0,1]$ leads to
		 $$\|H_t\|_1\le C_3\Psi_{F_0}(\vep)e^{	(C_1(t)+2C_0)t},\quad t\in[0,\infty).$$
		Using the representation of $C_1(t)$, we can choose some constant $C,c$ appropriately such that
		\be \|H_t\|_1\le C\Psi_{F_0}(\vep)e^{	e^{ct}},\quad t\in[0,\infty).\lb{H9}\ee
		Where $C,c$ denpend only on $N(F_0), E(F_0),N(G_0),E(G_0)$,  $a_0$,$b_0$,$\eta$,$q_1$,$R$,$M_{-1/2}(F_0)$, $M_{-1/2}(G_0)$. Finally if $\|H_0\|_1>0$ then taking $\vep=\|H_0\|_1$ in (\ref{H9}) gives (\ref{H0}).
		If $\|H_0\|_1=0$, then in (\ref{H9}) letting
		$\vep\to 0^+$ we conclude $\|H_t\|_1=0$ for all $t\in[0, T]$ and thus (\ref{H0})
		still holds true.
		This proves (\ref{H0}) for the case where $F_t$ has the moment production (\ref{Mprod}).The general case is still true since if $F_0=G_0$, $F_t$ has the moment production, then (\ref{H0}) tells us $F_t=G_t$ for all $t\in [0,\infty)$.
		$\hfill\Box$
\vskip2mm

As did for the classical Boltzmann equation,the collision integral $Q(f)$ can be decomposed as positive and negative parts:
\bes&& Q(f)(x)=Q^{+}(f)(x)-Q^{-}(f)(x),\lb{Qdecomposition}\\
&&Q^{+}(f)(x)=\int_{{\mR}_{\ge 0}^2}W(x,y,z)f(y)f(z)(1+f(x_*+f(x))\sqrt{y}\sqrt{z}{\rm d}y{\rm
	d}z, \dnumber \lb{Qplus}\\
&&Q^{-}(f)(x)=f(x)L(f)(x),\dnumber \lb{Qminus}\\
&&L(f)(x)=\int_{{\mR}_{\ge 0}^2}W(x,y,z)[f(x_*)(1+f(y)+f(z))]\sqrt{y}\sqrt{z}{\rm d}y{\rm
	d}z.\dnumber \lb{L}\ees
By the fact that $W(x,y,z)\le 4b^2_0W_H(x,y,z)$, it is easy to deduce that  for any $0\le f\in L^1({\mathbb R}_{+},\sqrt{x}{\rm d}x)$, the function $x\mapsto L(f)(x)$ is well-defined and
\be 0 \le L(f)(x)\le  4b^2_0\big(\sqrt{x}N(f)+M_{1/2}(f)+2[M_{-1/2}(f)]^2\big).\lb{newf4}
\ee
Where the moments for a nonnegative measurable function $f$ on ${\mathbb R}_{\ge 0}$ are defined in consistent with the case of measures: $M_p(f)=M_p(F)$ with ${\rm d}F(x)
=f(x)\sqrt{x}{\rm d}x$, i.e.
\be M_{p}(f)=\int_{{\mathbb R}_{+}}x^pf(x)\sqrt{x}{\rm d}x,\qquad  p\in(-\infty,\infty).\lb{moment-2}\ee  We also denote
$N(f)=M_0(f),
E(f)=M_1(f)$. And notice that $M_{-1/2}(f)=\int_{{\mathbb R}_{+}}f(x){\rm d}x.$
The following proposition gives an exponential-positive representation (i.e. Duhamel's formula) for a class of mild solutions. It has been used in the above Example.

\begin{proposition} \lb{proposition 3-3}
	Suppose $B({\bf {\bf v-v}_*},\omega)$ satisfy  Assumption \ref{assp} with $\eta\ge \fr{3}{2}$. Let
	$0\le f_0\in L^1({\mathbb R}_{\ge 0})$ have finite mass and energy. There exists a unique  conservative mild solution
	$f\in C([0, \infty); L^1({\mathbb R}_{\ge 0}))$
	 of Eq.(\ref{Equation1}) on $[0, \infty)$  satisfying
	$f(\cdot,0)=f_0$. Then there is a null set $Z\subset {\mathbb R}_{\ge 0}$ such that
	for all $x\in {\mathbb R}_{+}\setminus Z$ and all $t\in[0, \infty)$
	\be f(x,t)=f_0(x) e^{-\int_{0}^{t}L(f)(x,\tau){\rm d}\tau}
	+\int_{0}^{t}Q^{+}(f)(x,\tau)e^{-\int_{\tau}^{t}L(f)(x,s){\rm d}s}
	{\rm d}\tau\lb{Duhamel}\ee
	where $Q^{+}(f), L(f)$  are
	defined in (\ref{Qplus})-(\ref{L}).
\end{proposition}

\noindent\begin{proof}Since $f_0\in L^1({\mathbb R}_{\ge 0})$ means $M_{-1/2}(f_0)<\infty$. So using Theorem \ref{theorem1.10}, Theorem \ref{theorem1.11} we know there there exists a unique  conservative mild solution
	$f\in C([0, \infty); L^1({\mathbb R}_{\ge 0}))$
	of Eq.(\ref{Equation1}) on $[0, \infty)$  satisfying
	$f(\cdot,0)=f_0$. By definition of mild solutions and $Q(f)=Q^{+}(f)-fL(f)$ there is a null set $Z\subset {\mathbb R}_{\ge 0}$ which is independent of $t$
	such that for every $x\in {\mathbb R}_{\ge 0}\setminus Z$
	\be \fr{\p }{\p t}f(x,t)=Q^{+}(f)(x,t)-f(x,t)L(f)(x,t)
	\lb{diff}\ee
	for almost every $t\in[0,  \infty)$.
	Applying (\ref{newf4}) and $f\in  C([0,  \infty); L^1({\mathbb R}_{\ge 0}))$  we have \be
	\sup_{t\in [0, T]}L(f)(x,t)\le  \sup_{t\in [0, T]}( \sqrt{x}N(f(t))+M_{1/2}(f(t))+2\|f(t)\|_{L^1}^2 )<\infty\lb{loss1}\ee
	for all $T\in (0, \infty)$ and all $x>0$. Therefore, for every
	$x\in {\mathbb R}_{\ge 0}\setminus Z$,
	the function  $t\mapsto f(x,t)e^{\int_{0}^{t}L(f)(x,\tau){\rm d}\tau}$ is also  absolutely continuous on $[0, T]$ for all $T\in (0,  \infty)$ and thus
	the Duhamel's formula (\ref{Duhamel})
	follows from the differential equation (\ref{diff}).
\end{proof}
\par
We give a $L^{\infty}$ estimates for bounded mild solutions to end this section.
\begin{proposition} \lb{proposition 3-4}
		Suppose $B({\bf {\bf v-v}_*},\omega)$ satisfy Assumption \ref{assp} with $\eta\ge \fr{3}{2}$. Let $0\le f_0\in L^1({\mathbb R}_{\ge 0})$ have finite mass and energy and let
	$f\in C([0, \infty); L^1({\mathbb R}_{\ge 0}))$
	be the unique conservative mild solution of Eq.(\ref{Equation1}) on $[0,\infty)$  satisfying
	$f(\cdot,0)=f_0$. Suppose in addition $f_0\in L^{\infty}({\mathbb R}_{\ge 0})$. Then
	$f(\cdot, t)\in L^{\infty}({\mathbb R}_{\ge 0})$ for all $t\in [0, \infty)$ and
	there holds the following estimate: for all  $t\in[0, \infty)$,
	\bes\|f(t)\|_{L^{\infty}}
	\!\! &\le &\!\! (1+\|f_0\|_{L^{\infty}})\exp\Big(8b_0^2\int_{0}^{t}\|f(\tau)\|_{L^1}^2
	{\rm d}\tau\Big),\lb{finfty1}\ees
	\end{proposition}

\noindent\begin{proof}Let $K(t)$ be the right hand side of (\ref{finfty1}), i.e.
	$$K(t):=(1+\|f_0\|_{L^{\infty}})e^{2\int_{0}^{t}a(\tau){\rm d}\tau}, \quad a(t):=4b_0^2\|f(t)\|_{L^1}^2,\quad  t\in[0, \infty).$$
	By definition of mild solutions \textit{}and $f(x,0)=f_0(x)\le K(0)$ for all $x\in {\mathbb R}_{\ge 0}\setminus Z$ (here and below $Z\subset {\mathbb R}_{\ge 0}$ denotes any null set which is independent of time variable) we have for all $t>0$
	$$
	(f(x,t)-K(t))_{+}=\int_{0}^{t}( Q(f)(x,\tau)-2K(\tau)a(\tau) ){\bf 1}_{\{f(x,\tau)>K(\tau)\}}{\rm d}\tau.$$
	Taking integration with respect to $x\in{\mathbb R}_{+}$
	and omitting the negative part ${\cal Q}^{-}(f)\ge 0$ gives
	\bes
	\int_{{\mathbb R}_{\ge 0}}(f(x,t)-K(t))_{+}
	{\rm d}x \!\!&\le&\!\!
	\int_{0}^{t}{\rm d}\tau\int_{{\mathbb R}_{\ge 0}}Q^{+}(f)(x,\tau){\bf 1}_{\{f(x,\tau)>K(\tau)\}}
	{\rm d}x\nonumber \\
	\!\! & - &\!\!\int_{0}^{t}{\rm d}\tau\int_{{\mathbb R}_{\ge 0}}
	2K(\tau)a(\tau){\bf 1}_{\{f(x,\tau)>K(\tau)\}}
	{\rm d}x.\nonumber\ees
	For the integrand $Q^{+}(f)(x,\tau)$, we have
	$$ f(y,\tau)f(z,\tau)(1+f(x_*,\tau)+f(x,\tau))\le f(y,\tau)f(z,\tau)(f(x_*,\tau)-K(\tau))_{+}$$
	$$+f(y,\tau)f(z,\tau)(f(x,\tau)-K(\tau))_{+}+2K(\tau)f(y,\tau)f(z,\tau).
	$$
  Using the fact that $W(x,y,z)\sqrt{y}\sqrt{z}\le 4b^2_0$, we can obtain that
	\bes \int_{0}^{t}{\rm d}\tau\int_{{\mathbb R}_{\ge 0}}Q^{+}(f)(x,\tau){\bf 1}_{\{f(x,\tau)>K(\tau)\}}{\rm d}x\!\!
	&\le& \!\!\!\int_{0}^{t}2a(\tau){\rm d}\tau
	\int_{{\mathbb R}_{\ge 0}}
	( f(x,\tau)-K(\tau) )_{+}
	{\rm d}x\quad\nonumber\\
	\!\! &+&\!\!\!
	\int_{0}^{t}{\rm d}\tau\int_{{\mathbb R}_{\ge 0}}2K(\tau)a(\tau){\bf 1}_{\{f(x,\tau)>K(\tau)\}}{\rm d}x.\quad \nonumber\ees
  It follows that for all $t\in[0,\infty)$
	\beas&&
	\int_{{\mathbb R}_{\ge 0}}(f(x,t)-K(t))_{+}
	{\rm d}x\le
	\int_{0}^{t}2a(\tau){\rm d}\tau
	\int_{{\mathbb R}_{\ge 0}}
	(f(x,\tau)-K(\tau))_{+}
	{\rm d}x.\eeas
	By Gronwall inequality  we conclude
	$\int_{{\mathbb R}_+}(f(x,t)-K(t))_{+}
	{\rm d}x=0$ for all $t\in[0,\infty).$
	This implies $f(\cdot,t)\in L^{\infty}({\mathbb R}_{\ge 0})$ and
	$\|f(t)\|_{L^{\infty}}\le K(t)$ for all $t\in[0,\infty)$,  i.e. (\ref{finfty1}) holds true.
\end{proof}
\vskip4mm

{\bf Acknowledgement.}  This
work was partially supported by National Natural Science Foundation of China Grant No.11771236.

\end{document}